\documentclass[10pt,reqno]{amsart}
\usepackage{latexsym,amsmath}
\usepackage{enumerate}
\usepackage{amsfonts}
\usepackage{amssymb}
\usepackage{latexsym}

\usepackage[T1]{fontenc}
\usepackage{fourier}
\usepackage{bbm}

\usepackage{color}
\usepackage{fullpage}

\newcommand\step{\Sigma}

\newcommand\vQ{\vec Q}
\newcommand\reroot{\uparrow}
\newcommand\marg{\downarrow}
\newcommand\nix{\,\cdot\,}
\newcommand\vV{\vec V}
\newcommand\vS{\vec S}
\newcommand\ind{\mathrm{ind}}

\newcommand\atom{\delta}
\newcommand\thet{\vartheta}
\newcommand\ism{\cong}
\newcommand\G{\vec G}
\newcommand\T{\vec T}
\numberwithin{equation}{section}
\def\vec#1{\mathchoice{\mbox{\boldmath$\displaystyle#1$}}
{\mbox{\boldmath$\textstyle#1$}}
{\mbox{\boldmath$\scriptstyle#1$}}
{\mbox{\boldmath$\scriptscriptstyle#1$}}}
\newcommand\SIGMA{\vec\sigma}
\newcommand\TAU{\vec\tau}
\newtheorem{definition}{Definition}[section]
\newtheorem{claim}[definition]{Claim}
\newtheorem{example}[definition]{Example}
\newtheorem{remark}[definition]{Remark}
\newtheorem{theorem}[definition]{Theorem}
\newtheorem{lemma}[definition]{Lemma}
\newtheorem{proposition}[definition]{Proposition}
\newtheorem{corollary}[definition]{Corollary}
\newtheorem{fact}[definition]{Fact}
\newcommand\fG{\mathfrak{G}}
\newcommand\fT{\mathfrak{T}}
\newcommand\cutm{\Delta_{\Box}}
\newcommand\Cutm{D_{\Box}}
\newcommand\cB{\mathcal{B}}
\newcommand\cF{\mathcal{F}}
\newcommand\cE{\mathcal{E}}
\newcommand\cS{\mathcal{S}}
\newcommand\cL{\mathcal{L}}
\newcommand\cM{\mathcal{M}}
\newcommand\cP{\mathcal{P}}
\newcommand\cX{\mathcal{X}}
\newcommand\cV{\mathcal{V}}
\newcommand\cW{\mathcal{W}}
\newcommand\eps{\varepsilon}
\newcommand\ZZ{\mathbb{Z}}
\newcommand\NN{\mathbb{N}}
\newcommand\Var{\mathrm{Var}}
\newcommand\Erw{\mathrm{E}}
\newcommand{\vecone}{\vec{1}}
\newcommand{\Po}{{\rm Po}}
\newcommand{\Bin}{{\rm Bin}}
\newcommand{\Be}{{\rm Be}}
\newcommand\TV[1]{\left\|{#1}\right\|_{1}}
\newcommand\bc[1]{\left({#1}\right)}
\newcommand\cbc[1]{\left\{{#1}\right\}}
\newcommand{\bck}[1]{\left\langle{#1}\right\rangle}
\newcommand\brk[1]{\left\lbrack{#1}\right\rbrack}

\newcommand\abs[1]{\left|{#1}\right|}
\newcommand\RR{\mathbb{R}}
\newcommand{\stacksign}[2]{{\stackrel{\mbox{\scriptsize #1}}{#2}}}
\newcommand{\tensor}{\otimes}

\newcommand{\Erdos}{Erd\H{o}s}
\newcommand{\Renyi}{R\'enyi}
\newcommand{\Lovasz}{Lov\'asz}
\newcommand{\Szemeredi}{Szemer\'edi}
\newcommand\pr{\mathrm{P}} 
\newcommand\Lem{Lemma}
\newcommand\Prop{Proposition}
\newcommand\Thm{Theorem}

\newcommand\Cor{Corollary}
\newcommand\Sec{Section}
\newcommand\Chap{Chapter}

\begin{document}

\title{Limits of discrete distributions and Gibbs measures on random graphs}

\author[Coja-Oghlan, Perkins, Skubch]{Amin Coja-Oghlan$^{*}$, Will Perkins, Kathrin Skubch}
\thanks{$^{*}$The research leading to these results has received funding from the European Research Council under the European Union's Seventh 
Framework Programme (FP7/2007-2013) / ERC Grant Agreement n.\ 278857--PTCC}

\address{Amin Coja-Oghlan, {\tt acoghlan@math.uni-frankfurt.de}, Goethe University, Mathematics Institute, 10 Robert Mayer St, Frankfurt 60325, Germany.}

\address{Will Perkins, {\tt math@willperkins.org}, School of Mathematics, University of Birmingham, Edgbaston, Birmingham, UK.}

\address{Kathrin Skubch, {\tt skubch@math.uni-frankfurt.de}, Goethe University, Mathematics Institute, 10 Robert Mayer St, Frankfurt 60325, Germany.}

\begin{abstract}
\noindent
Building upon the theory of graph limits and the Aldous-Hoover representation
and inspired by Panchenko's work on asymptotic Gibbs measures [Annals of Probability 2013],
we construct continuous embeddings of discrete probability distributions.
We show that the theory of graph limits induces a meaningful notion of convergence and derive a corresponding version of the \Szemeredi\ regularity lemma.
Moreover, complementing recent work [Bapst et.\ al.\ 2015], we apply these results to Gibbs measures induced by sparse random factor graphs and verify the ``replica symmetric solution'' predicted in the physics literature under the assumption of non-reconstruction.

\end{abstract}

\maketitle

\section{Introduction}\label{Sec_intro}

\noindent
The systematic study of limits of discrete structures such as graphs or hypergraphs emerged about a decade ago~\cite{BCLSV,Lovasz}.
It has since become a prominent and fruitful area of research, with numerous applications in combinatorics and beyond.
The basic idea is to embed discrete objects into a ``continuous'' space so that tools from analysis, topology and measure theory can be brought to bear.
Conversely, the connection extends genuinely combinatorial ideas such as the \Szemeredi\ regularity lemma to the continuous world.

In this paper we study ``analytic embeddings'' of  probability distributions on discrete cubes, i.e.,
sets of the form $\Omega^V$ for some fixed finite set $\Omega$ and a large finite set $V$.
Arguably, probability measures on discrete cubes are among the most important and most basic objects in combinatorics, computer science and mathematical physics.
For instance, they occur as the Gibbs measures of finite spin systems such as the Ising model~\cite{MM}.
In this case $\Omega=\{\pm1\}$ and $V$ is a finite set of lattice points, whose size ultimately goes to infinity in the ``thermodynamic limit''.
Similarly, much of the theory of Markov Chain Monte Carlo deals with the correlations that, e.g., the uniform distribution
on the set of $k$-colorings of some graph $G=(V,E)$ induces~\cite{Levin}.
Thus, $\Omega$ would be the set $[k]=\{1,\ldots,k\}$ and $V$ is the vertex set of a (large) graph.
We are going to construct limiting objects of such measures, state a corresponding regularity lemma and illustrate applications to random graphs.

Perhaps the most prominent related construction is the Aldous--Hoover representation of exchangeable arrays~\cite{Aldous,Hoover}.
Its connection to graph limits was observed by Diaconis and Janson~\cite{DJ}.
Furthermore, Panchenko~\cite{Panchenko} used the Aldous-Hoover representation to introduce the notion of ``asymptotic Gibbs measures'',
the protagonists of his work on ``mean-field models'' such as the Sherrington--Kirkpatrick model.
Indeed, Panchenko has a proof of the Aldous-Hoover result from graph limits~\cite[Appendix~A]{PanchenkoBook}.
Our embedding of discrete measures into continuous space can be seen as a generalization of this approach
to arbitrary measures on discrete cubes.

The main contributions of the present work are as follows.
First, in \Sec~\ref{Sec_cubes} we construct a natural embedding of general measures on discrete cubes into a ``continuous'' metric space.
We highlight connections to the Aldous--Hoover representation and the theory of graph limits.
Moreover, we state a ``regularity lemma'' for such general measures that is a bit stronger than the regularity lemma of Bapst and Coja-Oghlan~\cite{Victor}.
In a sense, the main point of \Sec~\ref{Sec_cubes} is to study discrete probability measures using some of the main ideas from the theory of graph limits.
Second, building upon ideas of Panchenko~\cite{Panchenko}, in \Sec~\ref{Sec_Gibbs} we apply the concepts of \Sec~\ref{Sec_cubes}
to Gibbs measures induced by sparse random graphs.
In particular, we verify the ``replica symmetric solution'' predicted in the physics literature~\cite{pnas}
under the assumption that the sequence Gibbs measures converges in probability as the size $n$ of the random graph tends to infinity.
Additionally, we will see that the sequence of Gibbs measures induces a ``geometric'' limiting object 
 on the (infinite) Galton-Watson tree that describes the local structure of the sparse random graph.
Further, we show that a spatial mixing property called non-reconstruction is a sufficient condition for ``replica symmetry''.

Many of the proofs build upon known methods,  although the proofs in \Sec~\ref{Sec_Gibbs} require quite a bit of technical work.
Specific references will be given as we proceed, but let us point to \Lovasz' comprehensive treatment~\cite{Lovasz} of graph limits
and to Janson's work~\cite{Janson} that provides some of the measure-theoretic foundations.

\subsection*{Notation and preliminaries}
For a measurable space $(\Omega,\cF)$ we denote the set of probability measures by $\cP(\Omega,\cF)$ or briefly by $\cP(\Omega)$.
Moreover, for $x\in\Omega$ we denote by $\atom_x\in\cP(\Omega)$ the Dirac measure on $x$.
Further, we write $\lambda(\nix)$ for the Lebesgue measure on $\RR$.
If $\Omega$ is a finite set, then the $\sigma$-algebra is always understood to be its power set.
Given $\mu\in\cP(\Omega^n)$, we write $\SIGMA,\SIGMA',\ldots$ for independent samples from $\mu$.
Additionally, if $X:\Omega^n\to\RR$ is a function, then $\bck{X(\SIGMA)}_\mu=\sum_{\sigma\in\Omega^n}X(\sigma)\mu(\sigma)$
signifies the mean of $X$.

We often use the following notation to define a probability measure $p$ on a finite set $\Omega$.
If $f:\Omega\to[0,1]$ is a function that is not identically $0$, then
	$p(\omega)\propto f(\omega)$
is a shorthand for $$p(\omega)=f(\omega)\big/\sum_{\omega'\in\Omega}f(\omega').$$

We shall frequently work with the spaces $L_1(U,\RR^l)$ for a measurable $U\subset\RR^k$ (with $k,l$ natural numbers).
Recall that this is the space of measurable functions $f$ such that $\int|f|<\infty$, up to equality almost everywhere.
We tacitly identify the elements of $L_1(U,\RR^l)$ with specific (fixed) representatives, i.e., measurable functions $U\to\RR^l$, so that we can write $f(x)$ for $x\in U$.
Moreover, we remember that a {\em Polish space} is a complete metric space that has a countable dense subset.
Examples include the spaces $\RR^k$ and $L_1(U,\RR^l)$ for any $k,l$.

If $\mu\in\cP(\Omega)$ and $\nu\in\cP(\Omega')$, then we write $\Gamma(\mu,\nu)$ for the set of all couplings of $\mu,\nu$.
Thus, $\Gamma(\mu,\nu)$ is the set of all a probability measures $\gamma\in\cP(\Omega\times\Omega')$
such that $(x,y)\in\Omega\times\Omega'\mapsto x$ maps $\gamma$ to $\mu$ and
$(x,y)\in\Omega\times\Omega'\mapsto y$ maps $\gamma$ to $\nu$.
Moreover, $S_{[0,1)}$ is the set of all measurable $f:[0,1)\to[0,1)$ such that $f(\lambda)=\lambda$.
In addition, for a random variable $X:(\Omega,\mu)\to\Omega'$ we write $\cL(X)=X(\mu)$ for the distribution of $X$,
i.e., the probability measure on $\Omega'$ defined by $A'\mapsto\mu(X^{-1}(A'))$ for measurable $A'\subset\Omega'$.
Finally, if $\mu$ is a probability measure on a product space $\Omega^V$ and $U\subset V$, then $\mu_{\marg U}$
signifies the marginal distribution of $\mu$ on the coordinates $U$.
If $U=\cbc u$ is a singleton, then we write $\mu_{\marg u}=\mu_{\marg\{u\}}$.

We recall that
for probability measures $\mu,\nu$ defined on a metric space $E$  with metric $D(\nix,\nix)$ the
{\em $L_1$-Wasserstein metric} is defined as
	$$d_1(\mu,\nu)=\inf\cbc{\int_{\Omega\times\Omega}D(x,y)\, d\gamma(x,y):\gamma\in\Gamma(\mu,\nu)}.$$
For probability measures on a compact Polish space this metric induces the topology of weak convergence.

We are going to use the following well-known property of the Poisson distribution, and call it the ``Chen-Stein'' property after \cite{chen}.
If $X$ has distribution $\Po(d)$ for some $d>0$ and if $f(X)$ is a function such that $\Erw[X|f(X)|]<\infty$, then
	\begin{align}\label{eqmemoryless}
	\Erw[X f(X)]=d\Erw[f(X+1)]
	\end{align}

\section{Probability measures on cubes}\label{Sec_cubes}

\subsection{The cut metric}

Fix a finite set $\Omega$ and let $n\geq1$ be an integer.
We would like to identify a measure $\mu\in\cP(\Omega^n)$ with a ``continuous object''.
To this end, we represent a point $\sigma=(\sigma_1,\ldots,\sigma_n)\in\Omega^n$ by the function
	\begin{align*}
	\hat\sigma:[0,1)\to\cP(\Omega),\qquad x\mapsto\sum_{i=1}^n\delta_{\sigma_i}\vecone\{x\in[(i-1)/n,i/n)\}.
	\end{align*}
Thus, $\hat\sigma$ is a step function whose value on the interval $[(i-1)/n,i/n)$ equals the Dirac measure $\delta_{\sigma_i}\in\cP(\Omega)$.
Because $\Omega$ is finite, $\cP(\Omega)$ is just a simplex of dimension $|\Omega|-1$ in $\RR^\Omega$.
The induced Borel $\sigma$-algebra turns $\cP(\Omega)$ into a Polish space.
Moreover, $\mu\in\cP(\Omega^n)$ corresponds to the probability measure $\hat\mu$
on $L_1([0,1],\RR^\Omega)$ defined by
	\begin{equation}\label{eqdisc2cont}
	\hat\mu=\bck{\atom_{\hat\SIGMA}}_\mu=\sum_{\sigma\in\Omega^n}\mu(\sigma)\atom_{\hat\sigma}.
	\end{equation}
Indeed, let $\step_\Omega=L_1([0,1),\cP(\Omega))\subset L_1([0,1),\RR^\Omega)$ be the space of all measurable functions $f:[0,1)\to\cP(\Omega)$
with values in $\cP(\Omega)$ up to equality almost surely.
Then $\step_\Omega$ is Polish and $\hat\mu\in\cP(\step_\Omega)$.
Clearly, the map $\mu\mapsto\hat\mu$ is one-to-one.
Further, by extension of the discrete notation, we denote the mean of a measurable $X$ on $\step_\Omega$
with respect to $\mu\in\cP(\step_\Omega)$ by
	$$\bck{X(\SIGMA)}_\mu=\int_{\step_\Omega}X(\sigma)\,d\mu(\sigma).$$
Then the continuous analogue of (\ref{eqdisc2cont}) reads
	$\mu=\bck{\atom_{\SIGMA}}_\mu$.

Following the theory of graph limits~\cite{Lovasz},
we define the {\em strong cut metric} as
	\begin{align}\label{eqmetric1}
	\Cutm(\mu,\nu)&=\inf_{\gamma\in\Gamma(\mu,\nu)}\ \sup_{B\subset\step_\Omega^2,U\subset[0,1)}	
		\TV{\int_{B}\int_U \sigma_x-\tau_x\, dx\, d\gamma(\sigma,\tau)}\qquad(\mu,\nu\in\cP(\step_\Omega)),
	\end{align}
where, of course, $B,U$ are understood to be measurable.
Additionally,
	\begin{align}\label{eqmetric2}
	\cutm(\mu,\nu)&=\inf_{\gamma\in\Gamma(\mu,\nu),s\in S_{[0,1)}}\ \sup_{B\subset\step_\Omega^2,U\subset[0,1)}	
		\TV{\int_{B}\int_U \sigma_x-\tau_{s(x)}\, dx\, d\gamma(\sigma,\tau)}\qquad(\mu,\nu\in\cP(\step_\Omega)),
	\end{align}
is the  {\em weak cut metric}.
The general results~\cite{Janson} imply

\begin{fact}\label{Fact_attained}
$\Cutm(\nix,\nix)$ is a metric and $\cutm(\nix,\nix)$ is a pseudo-metric on $\cP(\step_\Omega)$.
Moreover, the infima in (\ref{eqmetric1})--(\ref{eqmetric2}) are attained.
\end{fact}

Let us write $M_\Omega$ for the space $\cP(\step_\Omega)$ endowed with the metric $\Cutm(\nix,\nix)$.
Moreover, call $\mu,\nu\in M_\Omega$ {\em equivalent} if $\cutm(\mu,\nu)=0$ and write $\bar\mu$ for the equivalence class of $\mu\in M_\Omega$.
Then $\cutm(\nix,\nix)$ induces a metric on the space $\cM_\Omega$ of equivalence classes.
We shall see momentarily that $\cM_\Omega$ (essentially) coincides with the usual graphon space.
But let us first look at two examples.

\begin{example}
Let $\Omega=\{0,1\}$, $p=\Be(1/2)$ and let $\mu_n=p^{\tensor n}\in\cP(\Omega^n)$ be the uniform distribution on the Hamming cube.
Letting $\sigma:[0,1)\to\cP(\Omega)$, $x\mapsto p$ be constant,
we expect that $\mu_n$ converges to $\nu=\atom_{\sigma}$ as $n\to\infty$.
Indeed, because $\nu$ is a Dirac measure there is just one coupling $\gamma$ of $\mu,\nu$.
Hence, 
the set $B$ in (\ref{eqmetric1}) really boils down to a set
$B_0\subset\Omega^n$ of configurations and the set $U\subset[0,1]$ to a weight function $u:[n]\to[0,1/n]$ such that
	\begin{align}\label{eqex1}
	\int_B\int_Up-\sigma(x)\,dx\,d\gamma(\sigma,\tau)&=p-\sum_{\sigma\in B_0}\sum_{x=1}^n u(x)\sigma(x).
	\end{align}
Since for $\SIGMA$ chosen from $\mu$ the sum $\sum_{x=1}^n u(x)\SIGMA(x)$ comprises of independent summands,
Azuma's inequality shows that for there is a constant $c>0$ such that for any $t>0$
	$\textstyle\pr\brk{\TV{p-\sum_{x=1}^n u(x)\SIGMA(x)}>t/\sqrt n}<2\exp(-ct^2).$
Therefore, the norm of (\ref{eqex1}) is $O(n^{-1/2})$ for all $B_0,U$.
\end{example}

\begin{example}
Let $\Omega=\{0,1\}$, $p=\Be(1/3)$, $q=\Be(2/3)$ and $\mu=\frac12(p^{\tensor n/2}\tensor q^{\tensor n/2}+q^{\tensor n/2}\tensor p^{\tensor n/2})$ for
an even $n>1$.
Let $\sigma:[0,1)\to\cP(\Omega)$, $x\mapsto p\vecone\{x<1/2\}+q\vecone\{x\geq1/2\}$
and $\tau:[0,1)\to\cP(\Omega)$, $x\mapsto \sigma(1-x)$.
Let $\nu=\frac12(\atom_\sigma+\atom_\tau)$.
Then 
	\begin{equation}\label{eqex2}
	\cutm(\mu,\nu)=\Cutm(\mu,\nu)=O(n^{-1/2}).
	\end{equation}
Indeed, to construct a coupling $\gamma$ of $\mu,\nu$ let $X,Y,Y'$ be three independent random variable such that $X=\Be(1/2)$,
$Y\in\{0,1\}^n$ has distribution $p^{\tensor n/2}\tensor q^{\tensor n/2}$ and $Y'\in\{0,1\}^n$
has distribution $p^{\tensor n/2}\tensor q^{\tensor n/2}$.
Further, let
	$G=(Y,\atom_\sigma)$ if $X=0$ and $G=(Y',\atom_\tau)$ otherwise
and let $\gamma$ be the law of $G$.
A similar application of Azuma's inequality as in the previous example yields (\ref{eqex2}).
\end{example}

\subsection{Alternative descriptions}

We recall that the (bipartite, decorated version of the) cut metric on the space $W_\Omega$ of measurable maps $[0,1)^2\to\cP(\Omega)$ can be defined as
	\begin{align*}
	\delta_{\Box}(f,g)&=\inf_{s,t\in S_{[0,1)}}\sup_{U,V\subset[0,1)}\TV{\int_{U\times V}f(x,y)-g(s(x),t(y))\,dx\,dy}\qquad
			\mbox{(cf.~\cite{Janson,Lovasz,LovaszSzegedy}).}
	\end{align*}
Let $\cW_\Omega$ be the space obtained from $W_\Omega$ by identifying $f,g\in W_\Omega$ such that $\delta_{\Box}(f,g)=0$.
Applying \cite[\Thm~7.1]{Janson} to our setting, we obtain

\begin{proposition}\label{Prop_homeomorphic}
There is a homeomorphism $\cM_\Omega\to\cW_\Omega$.
\end{proposition}
\begin{proof}
We recall that for any $\mu\in\cP(\step_\Omega)$ there exists a measurable $\varphi:[0,1)\to\step_\Omega$
such that $\mu=\varphi(\lambda)$, i.e., $\mu(A)=\lambda(\varphi^{-1}(A))$ for all measurable $A\subset\step_\Omega$.
Hence, recalling that $\varphi(x)\in L_1([0,1),\cP(\Omega))$,
	$\mu$ yields a graphon $w_\mu:[0,1]^2\to\cP(\Omega)$, $(x,y)\mapsto(\varphi(x))(y)$.
Due to~\cite[\Thm~7.1]{Janson} the map $\bar\mu\in\cM_\Omega\mapsto w_\mu\in\cW_\Omega$ is a homeomorphism.
\end{proof}

\begin{corollary}
$\cM_\Omega$ is a compact Polish space.
\end{corollary}
\begin{proof}
This follows from \Prop~\ref{Prop_homeomorphic} and the fact that $\cW_\Omega$ has these properties~\cite[\Thm~9.23]{Lovasz}.
\end{proof}

Diaconis and Janson~\cite{DJ} pointed out that the connection between $\cW_\Omega$ and the Aldous-Hoover representation of ``exchangeable arrays''
	(see also Panchenko~\cite[Appendix A]{PanchenkoBook}).
To apply this observation to $\cM_\Omega$,
recall that $\Omega^{\NN\times\NN}$ is compact (by Tychonoff's theorem) and that
a sequence $(A(n))_n$ of $\Omega^{\NN\times\NN}$-valued random variables converges to $A$ in distribution iff
	$$\lim_{n\to\infty}\pr\brk{\forall i,j\leq k:A_{ij}(n)=a_{ij}}=\pr\brk{\forall i,j\leq k:A_{ij}=a_{ij}}
	\qquad\mbox{for all $k$, $a_{ij}\in\Omega$}.$$
Now, for $\bar\mu\in\cM_\Omega$ define a random array $\vec A(\bar\mu)=(\vec A_{ij}(\bar\mu))\in\Omega^{\NN\times\NN}$ as follows.
Let $(\SIGMA_i)_{i\in\NN}$ be a sequence of independent samples from the distribution $\mu$,
independent of the sequence $(\vec x_i)_{i\in\NN}$ of independent uniform samples from $[0,1)$.
Finally, independently for all $i,j$ choose $\vec A_{ij}(\bar\mu)\in\Omega$ from the distribution $\SIGMA_i(\vec x_j)\in\cP(\Omega)$.
Then in our context the correspondence from~\cite[\Thm~8.4]{DJ} reads

\begin{corollary}\label{Cor_AldousHoover}
The sequence $(\bar\mu_n)_n$ converges to $\bar\mu\in\cM_\Omega$ iff  $\vec A(\bar\mu_n)$ converges to $\vec A(\bar\mu)$ in distribution.
\end{corollary}

While \Cor~\ref{Cor_AldousHoover} characterizes convergence in  $\cutm(\nix,\nix)$, the following statement applies to the strong metric $\Cutm(\nix,\nix)$.
For $\sigma\in\step_\Omega$ and $x_1,\ldots,x_k\in[0,1)$ define $\sigma_{\marg x_1,\ldots,x_k}=\sigma(x_1)\tensor\cdots\tensor\sigma(x_k)\in\cP(\Omega^k)$.
Moreover, for $\mu\in M_\Omega$ let
	$$\mu_{\marg x_1,\ldots,x_k}=\int_{\step_\Omega}\sigma_{\marg x_1,\ldots,x_k}\,d\mu(\sigma).$$
If $\mu\in\cP(\Omega^n)$ is a discrete measure,  then
$\hat\mu_{\marg x_1,\ldots,x_k}=\widehat{\mu_{\marg i_1,\ldots,i_k}}$ with $i_j=\lceil nx_j\rceil$.
As before, we let $(\vec x_i)_{i\geq1}$ be a sequence of independent uniform samples from $[0,1)$.

\begin{corollary}\label{Cor_sampling}
If $(\mu_n)_n\stacksign{$\Cutm$}\to\mu\in M_\Omega$, then 
for any integer $k\geq1$ we have
	$\lim_{n\to\infty}\Erw\TV{\mu_{n\marg \vec x_1,\ldots,\vec x_k}-\mu_{\marg \vec x_1,\ldots,\vec x_k}}=0.$
\end{corollary}
\begin{proof}
By \cite[\Thm~8.6]{Janson}  we can turn $\mu,\mu_n$ into graphons $w,w_n:[0,1)^2\to\cP(\Omega)$ such that for all $n$
	\begin{align*}
	\mu&=\int_0^1 \atom_{w(\nix,y)}\,dy,\quad\mu_n=\int_0^1 \atom_{w_n(\nix,y)}\,dy\quad\mbox{and}\quad
		\Cutm(\mu,\mu_n)=\sup_{U,V\subset[0,1)}\TV{\int_{U\times V}w(x,y)-w_n(x,y) \,dx\,dy}.
	\end{align*}
Let $(\vec y_j)_{j\geq 1}$ be independent and uniform on $[0,1)$ and independent of $(\vec x_i)_{i\geq1}$.
By \cite[\Thm~10.7]{Lovasz}, we have $\lim_{n\to\infty}\Cutm(\mu_n,\mu)=0$ iff
	\begin{equation}\label{eq_samp}
	\lim_{r\to\infty}\limsup_{n\to\infty}\ 
		\Erw\brk{\max_{I,J\subset[r]}\TV{\sum_{(i,j)\in I\times J}
		w(\vec x_i,\vec y_j)-w_n(\vec x_i,\vec y_j)
			}}=0.
	\end{equation}
Hence, we are left to show that (\ref{eq_samp}) implies
\begin{equation}\label{eq_marg}
\forall k\geq1: \lim_{n\to\infty}\Erw\TV{\mu_{n\marg \vec x_1,\ldots,\vec x_k}-\mu_{\marg \vec x_1,\ldots,\vec x_k}} = 0.
\end{equation}
To this end, we note that by the strong law of large numbers uniformly for all $x_1,\ldots, x_k\in[0,1]$
and $n$,
\begin{align}\label{eq_LLN1}
 \frac 1r \sum_{j=1}^r (w(x_1,\vec y_j),\ldots,w(x_k,\vec y_j))&\ \stacksign{$r\to\infty$}\to\ \mu_{\marg x_1,\ldots, x_k}&\mbox{in probability},\\
  \frac 1r \sum_{j=1}^r (w_n(x_1,\vec y_j),\ldots,w_n(x_k,\vec y_j))&\ \stacksign{$r\to\infty$}\to\ \mu_{n\marg x_1,\ldots, x_k}&\mbox{in probability}.
  	\label{eq_LLN2}
\end{align}
Hence, if \eqref{eq_samp} holds, then (\ref{eq_marg}) follows from (\ref{eq_LLN1})--(\ref{eq_LLN2}).
\end{proof}

\noindent
As an application of \Cor~\ref{Cor_sampling} we obtain

\begin{corollary}\label{Cor_factorise}
Assume that $(\mu_n)_n$ is a sequence such that $\mu_n\stacksign{$\Cutm$}\to\mu\in M_\Omega$.
The following statements are equivalent.
\begin{enumerate}[(i)]
\item There is $\sigma\in\Sigma_\Omega$ such that $\mu=\atom_\sigma$.
\item For any integer $k\geq2$ we have
	\begin{equation}\label{eqFactorise}
	\lim_{n\to\infty}\Erw\TV{\mu_{n\marg\vec x_1,\ldots,\vec x_k}-\mu_{n\marg\vec x_1}\tensor\cdots\tensor\mu_{n\marg\vec x_k}}=0.
	\end{equation}
\item The condition (\ref{eqFactorise}) holds for $k=2$.	
\end{enumerate}
\end{corollary}
\begin{proof}
The implication (i)$\Rightarrow$(ii) follows from \Cor~\ref{Cor_sampling} and the step from (ii) to (iii) is immediate.
Hence, assume that (iii) holds.
Then by \Cor~\ref{Cor_sampling} and the continuity of the $\tensor$-operator,
	\begin{align}\label{eqFactorise0}
	\Erw\TV{\mu_{\marg\vec x_1,\vec x_2}-\mu_{\marg\vec x_1}\tensor\mu_{\marg\vec x_2}}&=
		\lim_{n\to\infty}\Erw\TV{\mu_{n\marg\vec x_1,\vec x_2}-\mu_{n\marg\vec x_1}\tensor\mu_{n\marg\vec x_2}}=0.
	\end{align}
Define $\tilde\sigma:[0,1)\to\cP(\Omega)$ by $x\mapsto\mu_{\marg x}$ and assume that $\mu\neq\atom_{\tilde\sigma}$.
Then $\Cutm(\mu,\atom_{\tilde\sigma})>0$ (by Fact~\ref{Fact_attained}), whence there exist $B\subset\step_\Omega$, $U\subset[0,1)$, $\omega\in\Omega$ such that
	\begin{align}\label{eqFactorise1}
	\int_B\brk{\int_U\sigma_x(\omega)-\tilde\sigma_x(\omega)\,dx}^2\,d\mu(\sigma)>0.
	\end{align}
However, (\ref{eqFactorise0}) entails
	\begin{align*}
	\int_{\step_\Omega}\brk{\int_U\sigma_x(\omega)-\tilde\sigma_x(\omega)\,dx}^2\,d\mu(\sigma)
		&=\int_{\step_\Omega}\int_U\int_U\sigma_x(\omega)\sigma_y(\omega)-\tilde\sigma_x(\omega)\tilde\sigma_y(\omega)\,dx\,dy\,d\mu(\sigma)\\
		&=\Erw[\mu_{\marg\vec x_1,\vec x_2}-\mu_{\marg\vec x_1}\tensor\mu_{\marg\vec x_2}|\vec x_1,\vec x_2\in U]=0,
	\end{align*}
in contradiction to (\ref{eqFactorise1}).
\end{proof}

\begin{remark}
Strictly speaking, the results from~\cite{DJ,Lovasz} are stated for graphons with values in $[0,1]$, i.e.,  $\cP(\Omega)$ for $|\Omega|=2$.
However, they extend to $|\Omega|>2$ directly.
For instance, the compactness proof~\cite[\Chap~9]{Lovasz} is by way of the regularity lemma, which we extend in \Sec~\ref{Sec_reg} explicitly.
Moreover, the sampling result for \Cor~\ref{Cor_sampling} follows from~\cite[\Chap~10]{Lovasz} by viewing $w:[0,1)^2\to\cP(\Omega)$
as a family $(w_\omega)_{\omega\in\Omega}$, $w_\omega:(x,y)\mapsto w_{x,y}(\omega)\in[0,1]$.
Finally, the proof of \Cor~\ref{Cor_AldousHoover} in~\cite{DJ} by counting homomorphisms, 
extends to $\cP(\Omega)$-valued graphons~\cite[\Sec~17.1]{Lovasz}.
\end{remark}

\subsection{Algebraic properties}
The cut metric is compatible with basic algebraic operations on measures.
The following is immediate.

\begin{fact}
If $\mu_n\stacksign{$\Cutm$}\to\mu$, $\nu_n\stacksign{$\Cutm$}\to\nu$,
then $\alpha\mu_n+(1-\alpha)\nu_n\stacksign{$\Cutm$}\to\alpha\mu+(1-\alpha)\nu$ for any $\alpha\in(0,1)$.
\end{fact}

The construction of a ``product measure'' is slightly more interesting.
Let $\Omega,\Omega'$ be finite sets.
For $\sigma\in\step_\Omega,\tau\in\step_{\Omega'}$ we define $\sigma\times\tau\in\step_{\Omega\times\Omega'}$ 
by letting $\sigma\times\tau(x)=\sigma(x)\tensor\tau(x)$,
where $\sigma(x)\tensor\tau(x)\in\cP(\Omega\times\Omega')$ is the usual product measure of $\sigma(x),\tau(x)$.
Further, for $\mu\in M_\Omega,\nu\in M_{\Omega'}$ we define $\mu\times\nu\in M_{\Omega\times\Omega'}$ by
	\begin{align*}
	\mu\times\nu&=\int_{\step_{\Omega}\times\step_{\Omega'}}\atom_{\sigma\times\tau}\,d\mu\tensor\nu(\sigma,\tau).
	\end{align*}
Clearly, $\mu\times\nu$ is quite different from the usual product measure $\mu\tensor\nu$.
However, for {\em discrete} measures we observe the following.

\begin{fact}
For $\mu\in\cP(\Omega^n)$ and $\nu\in\cP({\Omega'}^n)$ we have
	$\hat\mu\times\hat\nu=\widehat{\mu\tensor\nu}$.
\end{fact}

\begin{proposition}
If $\mu_n\stacksign{$\Cutm$}\to\mu\in M_\Omega$, $\nu_n\stacksign{$\Cutm$}\to\nu\in M_{\Omega'}$,
then $\mu_n\times\nu_n\stacksign{$\Cutm$}\to\mu\times\nu$.
\end{proposition}
\begin{proof}
Let $\eps>0$ and choose $n_0$ large enough so that
$\Cutm(\mu_n,\mu)<\eps$ and $\Cutm(\nu_n,\nu)<\eps$ for all $n>n_0$.
By Fact~\ref{Fact_attained} there exist couplings
$\gamma_n,\gamma_n'$ of $\mu_n,\mu$ and $\nu_n,\nu$ such that  (\ref{eqmetric1}) is attained.
Because  $\TV{p\tensor p'-q\tensor q'}\leq\TV{p-q}+\TV{q-q'}$ for any $p,q\in\cP(\Omega)$, $p',q'\in\cP(\Omega)$,
we obtain for any $U\subset[0,1)$, $B\subset M_\Omega$, $B'\subset M_\Omega$
	\begin{align*}
	\TV{\int_{B\times B'}\int_U\sigma\times\sigma'(x)-\tau\times\tau'(x)\,dx\,d\gamma_n\tensor\gamma_n'(\sigma,\tau,\sigma',\tau')}<2\eps,
	\end{align*}
as desired.
\end{proof}

\subsection{Regularity}\label{Sec_reg}
For $\sigma\in\Sigma_\Omega$ and $U\subset[0,1)$ measurable we write
	$$\sigma[\omega|U]=\int_U\sigma_x(\omega)\,dx.$$
Moreover, for $\mu\in M_\Omega$ and a measurable $S\subset\step_\Omega$ with $\mu(S)>0$ we let $\mu[\nix|S]\in M_\Omega$
be the conditional distribution.
Further, let $\vV=(V_1,\ldots,V_K)$ be a partition of $[0,1)$ into a finite number of pairwise disjoint measurable sets.
Similarly, let $\vS=(S_1,\ldots,S_L)$ be a partition of $\step_\Omega$ into pairwise disjoint measurable sets.
We write $\#\vV,\#\vS$ for the number $K,L$ of classes, respectively.
A measure $\mu\in M_\Omega$ is {\em $\eps$-regular} with respect to $(\vV,\vS)$
if there exists $R\subset[\#\vV]\times[\#\vS]$ such that the following conditions hold.
\begin{description}
\item[REG1] $\lambda(V_i)>0$ and $\mu(S_j)>0$ for all $(i,j)\in R$.
\item[REG2] $\sum_{(i,j)\in R}\lambda(V_i)\mu(S_j)>1-\eps$.
\item[REG3] for all $(i,j)\in R$ and all $\sigma,\sigma'\in S_j$ we have
	$\TV{\sigma[\nix|V_i]-\sigma'[\nix|V_i]}<\eps$.
\item[REG4] if $(i,j)\in R$, then for every $U\subset V_i$ with $\lambda(U)\geq\eps\lambda(V_i)$
	and every $T\subset S_j$ with $\mu(T)\geq\eps\mu(S_j)$ we have
		$$\TV{\bck{\SIGMA[\nix|U]}_{\mu[\nix|T]}-\bck{\SIGMA[\nix|V_i]}_{\mu[\nix|S_j]}}<\eps.$$
\end{description}

Thus, $R$ is a set of index pairs $(i,j)$ of ``good squares'' $V_i\times S_j$.
{\bf REG1} provides that  every good square has positive measure and {\bf REG2} that the total probability mass of good squares is at least $1-\eps$.
Further, by {\bf REG3} the averages $\sigma[\nix|V_i],\sigma'[\nix|V_i]\in\cP(\Omega)$ over $V_i$ of any two $\sigma,\sigma'\in S_j$ are close.
Finally, and most importantly, {\bf REG4} requires that the average $\bck{\SIGMA[\nix|U]}_{\mu[\nix|T]}$ over a ``biggish'' sub-square $U\times T$
is close to the mean over the entire square $V_i\times S_j$.

A {\em refinement} of a partition $(\vV,\vS)$ is a partition $(\vV',\vS')$ such that
for every pair $(i',j')\in[\#\vV']\times[\vS']$ there is a pair $(i,j)\in [\#\vV]\times[\vS]$ such that $(V_{i'}',S_{j'}')\subset(V_i,S_j)$.

\begin{theorem}\label{Thm_reg}
For any $\eps>0$ there exists $N=N(\eps,\Omega)$ such that for every $\mu\in M_\Omega$ the following is true.
Every partition $(\vV_0,\vS_0)$ with $\#\vV_0+\#\vS_0\leq1/\eps$ has a refinement
$(\vV,\vS)$ such that $\#\vV+\#\vS\leq N$
with respect to which $\mu$ is $\eps$-regular.
\end{theorem}

In light of \Prop~\ref{Prop_homeomorphic}, \Thm~\ref{Thm_reg} would follow from the regularity lemma for
	graphons~\cite[\Lem~9.16]{Lovasz} if we were to drop condition {\bf REG3}.
In fact, adapting the standard proof from~\cite{Szemeredi} to accommodate {\bf REG3} is not difficult.
For the sake of completeness we carry this out in detail in \Sec~\ref{Sec_Kathrin}.

A regularity lemma for measures on $\Omega^n$ was proved in~\cite{Victor}.
But even in the discrete case \Thm~\ref{Thm_reg} gives a stronger result.
The improvement is that {\bf REG4} above holds for all ``small sub-squares'' $U\times T$ simultaneously.

How does the concept of regularity connect with the cut metric?
For a partition $\vV$ of $[0,1]$ and $\sigma\in\step_\Omega$ define
	$\sigma[\nix|\vV]\in\cW_\Omega$ by
	$$\sigma_x[\omega|\vV]=\sum_{i\in[\#\vV]}\vecone\{x\in V_i\}\sigma_x[\omega|V_i].$$
Thus, $\sigma[\nix|\vV]:[0,1)\to\cP(\Omega)$ is constant on the classes of $\vV$.
Further, for a pair  $(\vV,\vS)$ of partitions and $\mu\in M_\Omega$ let
	$$\mu[\nix|\vV,\vS]=\sum_{i\in[\#\vS]}\atom_{\int_{S_i}\sigma[\nix|\vV]d\mu(\sigma)}.$$
Hence, $\mu[\nix|\vV,\vS]\in M_\Omega$ is supported on a discrete set of functions $[0,1)\to\cP(\Omega)$
that are constant on the classes of $\vV$.
We might think of $\mu[\nix|\vV,\vS]$ as the ``conditional expectation'' of $\mu$ with respect to $(\vV,\vS)$.

\begin{proposition}\label{Prop_reg2metric}
Let $\eps>0$ and assume that $\mu$ is $\eps$-regular w.r.t.\ $(\vV,\vS)$.
Then
	$\Cutm(\mu,\mu[\nix|\vV,\vS])<2\eps$.
\end{proposition}
\begin{proof}
Let $\sigma^{(i)}=\int_{S_i}\sigma[\nix|\vV]d\mu(\sigma)$.
We define a coupling $\gamma$ of $\mu,\mu[\nix|\vV,\vS]$ in the obvious way: for a measurable $X\subset S_i$ let 
	$\gamma(X\times\{\sigma^{(i)}\})=\mu(X)$.
Now, let $U\subset[0,1]$ and $B\subset\step_\Omega^2$ be measurable.
Due to the construction of our coupling we may assume that $B=\bigcup_i B_i\times\{\sigma^{(i)}\}$ for certain sets $B_i\subset S_i$.
Moreover, let $U_j=U\cap V_j$.
Then
	\begin{align*}
	\TV{\int_B\int_U\sigma(x)-\tau(x) dx d\eta(\sigma,\tau)}&\leq
	\sum_{(i,j):\mu(S_i)\lambda(V_j)>0}\mu(S_i)\lambda(V_j)\TV{\int_{B_i}\int_{U_j}\sigma(x)-s_i(x)\frac{dx}{\lambda(V_j)}\frac{d\mu(\sigma)}{\mu(S_i)}}.
	\end{align*}
By {\bf REG1} and {\bf REG4} the last expression is less than $2\eps$.
\end{proof}

\begin{corollary}
For any $\eps>0$ there exists $N=N(\eps)>0$ such that for any $\mu\in M_\Omega$ there exist  $\sigma_1,\ldots,\sigma_N\in\step_\Omega$
and $w=(w_1,\ldots,w_N)\in\cP([N])$ such that
	$\Cutm\bc{\mu,\sum_{i=1}^kw_i\atom_{\sigma_i}}<\eps.$
\end{corollary}
\begin{proof}
This is immediate from \Thm~\ref{Thm_reg} and \Prop~\ref{Prop_reg2metric}.
\end{proof}

\subsection{Proof of \Thm~\ref{Thm_reg}}\label{Sec_Kathrin}
Following the path beaten in~\cite{Victor,Szemeredi,Tao}, we define the index of $(\vV,\vS)$ as
	\begin{align*}
	\ind_\mu(\vV,\vS)&=\Erw\bck{\Var[\SIGMA_{\vec x}[\omega]|\vV,\vS]}_\mu
		=\frac1{|\Omega|}\sum_{\omega\in\Omega}\sum_{i=1}^{\#\vV}\sum_{j=1}^{\#S_j}
			\int_{S_j}\int_{V_i}\bc{\sigma_x(\omega)-\int_{S_j}\int_{V_i}\sigma_y(\omega)\frac{dy}{\lambda(V_i)}\frac{d\mu(\sigma)}{\mu(S_j)}}^2dxd\mu(\sigma).
	\end{align*}

\noindent
There is only one simple step that we add to the proof from~\cite{Szemeredi}.
Namely, following~\cite{Victor}, we begin by refining the partition $\vS_0$ to guarantee {\bf REG3}.
Specifically, the compact set $\cP(\Omega)$ has a 
partition into a finite number of sets $\vQ=(Q_1,\ldots, Q_K)$ such that $\TV{\mu - \mu'}\leq\eps$ for all $\mu, \mu'\in Q_i$, $i\in[K]$.
Now, let $\vV(1)=\vV_0$ and let $\vS(1)$ be the coarsest refinement of $\vS_0$ such that
for every $i\in[\#\vV_1]$, $j\in\#\vS_1$ there is $k\in[K]$ such that $\sigma[\nix|V_i]\in Q_k$ for all $\sigma\in S_j(1)$.
Then $\#\vS(1)\leq K\#\vS_0$.

Starting from $(\vV(1),\vS(1))$, we construct a sequence $(\vV(t),\vS(t))$ of partitions inductively.
The construction stops once $\mu$ is $\eps$-regular w.r.t.\ $(\vV(t),\vS(t))$, in which case we are done.
Assuming otherwise, consider the set 
$\bar R(t)\subset [\#\vV(t)]\times[\#\vS(t)]$ of $(i,j)$ such that {\bf REG4} fails to hold on $(V_i,S_j)$.
Then
	\begin{equation}
	 \sum_{(i,j)\in \bar R(t)}\lambda(V_i)\mu(S_j)\geq\eps.
	\end{equation}
Further, for each $(i,j)\in\bar R$ there exist $U_i\subset V_i$, $\lambda(U_i)\geq\eps\lambda(V_i)$, $T_j\subset S_j$, $\mu(T_j)\geq\eps\mu(S_j)$
and $\omega_{ij}$ such that
\begin{align}\label{eq_mass}
  \abs{\bck{\SIGMA[\omega_{ij}|U_i]}_{\mu[\nix|T_j]}-\bck{\SIGMA[\omega_{ij}|V_i]}_{\mu[\nix|S_j]}}
  \geq \frac{\eps}{|\Omega|}.
\end{align}
Obtain the partition $(\vV(t,i,j),\vS(t,i,j))$ from $(\vV(t),\vS(t))$ by splitting $V_i$ and $S_j$ into the sub-classes
$U_i$, $V_i\setminus U_i$ and $S_j$, $S_j\setminus T_j$. Clearly, $\#\vV_{ij}\leq2\#\vV$ and $\#\vS_{ij}\leq2\#\vS$, respectively.
Then
	\begin{align}\nonumber
	\Erw[\bck{\Var[\SIGMA_{\vec x}(\omega_{ij})|\vV(t),\vS(t)]}_{\mu[\nix|S_j]}|V_i]
		&=\Erw[\bck{\Var[\SIGMA_{\vec x}(\omega_{ij})|\vV(t,i,j),\vS(t,i,j)]}_{\mu[\nix|S_j]}|V_i]\\
		&\qquad+\Erw[\bck{\Var[\Erw[\SIGMA_{\vec x}(\omega_{ij})|\vV(t,i,j),\vS(t,i,j)]|\vV(t),\vS(t)]}_{\mu[\nix|S_j]}|V_i].
			\label{eqKS1}
	\end{align}
Moreover, \eqref{eq_mass} implies that on
$V_i\times S_j$ we have
\begin{equation}\label{eqKS2}
\Var[\Erw[\SIGMA_{\vec x}[\omega]|\vV(t,i,j),\vS(t,i,j)]|\vV(t),\vS(t)]\geq
\frac{\lambda(U_i)}{\lambda(V_i)}\frac{\mu(T_j)}{\mu(S_j)}\abs{\bck{\SIGMA[\omega_{ij}|U_i]}_{\mu[\nix|T_j]}-\bck{\SIGMA[\omega_{ij}|V_i]}_{\mu[\nix|S_j]}}^2.
\end{equation}
Combining (\ref{eq_mass}), (\ref{eqKS1}) and (\ref{eqKS2}), we obtain
	\begin{equation}\label{eq_ind}
	\Erw[\bck{\Var[\SIGMA_{\vec x}(\omega_{ij})|\vV(t),\vS(t)]}_{\mu[\nix|S_j]}|V_i]
		\geq\Erw[\bck{\Var[\SIGMA_{\vec x}(\omega_{ij})|\vV(t,i,j),\vS(t,i,j)]}_{\mu[\nix|S_j]}|V_i]+
 \lambda(U_i)\mu(T_j)\frac{\eps^2}{|\Omega|^3}.
\end{equation}
Now, let $(\vV(t+1),\vS'(t+1))$ denote the coarsest common refinement of all the partitions $(\vV(t,i,j),\vS(t,i,j))_{i,j\in\bar R}$.
Finally, obtain $\vS(t+1)$ from $\vS'(t+1)$ as in the very first step by splitting each class $\vS_j'(t+1)$
into classes $(\vS_{j,k}'(t+1))_{k\leq K}$ such that $\TV{\sigma[\nix|\vV_i(t+1)]-\sigma'[\nix|\vV_i(t+1)]}<\eps$ for all $i\in[\#\vV(t+1)]$.
Then (\ref{eq_ind}) and the monotonicity of the conditional variance imply
\begin{align}\label{eq_ind2}
 \ind_\mu(\vV(t+1),\vS(t+1))
 &\leq \ind_{\mu}(\vV(t),\vS(t))-\frac{\eps^2}{|\Omega|^3}\sum_{(i,j)\in\bar R}\lambda(U_i)\mu(T_j)
 \leq \ind_{\mu}(\vV(t),\vS(t))-\frac{\eps^5}{|\Omega|^3}.
\end{align}
Since the index lies between $0$ and $1$, (\ref{eq_ind2}) implies that the construction stops
after at most $\eps^{-5}|\Omega|^3$ steps.

\section{The replica symmetric solution}\label{Sec_Gibbs}

\noindent
In this section we apply the notion of convergence and the results from \Sec~\ref{Sec_cubes} to Gibbs measures induced by random graphs.
Many of the arguments build upon the work of Panchenko on asymptotic Gibbs measures~\cite{Panchenko}.

\subsection{Random factor graphs}
A remarkably wide variety of problems in combinatorics can be described in terms of {\em factor graphs}.
These are bipartite graphs with two types of vertices called {\em variable nodes} and {\em constraint nodes}.
The variable nodes can be assigned ``spins'' from a finite set $\Omega$ and each constraint node is decorated with a weight function
that assigns every spin configuration of its adjacent variable nodes a positive weight.
Natural examples of such models occur in combinatorics, mathematical physics or information theory~\cite{MM,Rudi}.
We shall see a few concrete examples in just a moment.

Let us first attempt an abstract, fairly comprehensive definition (see~\cite{Victor} for an even more general setup).
Suppose that $\Omega$ is a given finite set of {\em spins} and that $\Psi$ is a set of functions $\psi:\Omega^{k_{\psi}}\to(0,\infty)$ of arity $k_\psi\geq1$.
A {\em $(\Omega,\Psi)$-factor graph} $G=(V_G,F_G,(\psi_a)_{a\in F_G},\partial_G)$ consists of
\begin{description}
\item[GR1] (finite or) countable disjoint sets $V_G,F_G$, 
\item[GR2] a map $a\in F_G\mapsto\psi_a\in\Psi$,
\item[GR3] a map
	$\partial_G:F_G\to\bigcup_{l\geq 1}V_G^l$ such that $\partial_Ga=(\partial_{G}(a,j))_j\in V_G^{k_{\psi_a}}$ for all $a\in F_G$
	and such that for every $x\in V_G$ the set $\{a\in F_G:\exists j\in[k_{\psi_a}]:x=\partial_{G}(a,j)\}$ is finite.
\end{description}
Let us introduce the shorthand $k_a=k_{\psi_a}$ for the arity of the constraint $a$.
Moreover, if $\sigma:V_G\to\Omega$, then we let 
$\sigma(\partial_Ga)=(\sigma(\partial_G(a,1)),\ldots,\sigma(\partial_G(a,k_{a})))$.

A finite factor graph $G$ naturally induces a probability measure on the set $\Omega^{V_G}$ of all possible assignments:
 the {\em Gibbs measure} of $G$ is defined by
	\begin{align}\label{eqGibbs1}
	\mu_G&:\Omega^{V_G}\to(0,1),\quad\sigma\mapsto Z_G^{-1}
		\prod_{a\in F_G}\psi_a(\sigma(\partial_Ga)),&\mbox{where}\qquad
	Z_G&=\sum_{\sigma\in\Omega^{V_G}}\prod_{a\in F_G}\psi_a(\sigma(\partial_Ga))
	\end{align}
is the {\em partition function} of $G$.
Thus, the probability mass that $\mu_G$ assigns to $\sigma$ is proportional to the weights
 $\psi_a(\sigma(\partial_Ga))$  that the constraint nodes $a$ assign to the spin configurations of their incident variables.

Naturally, we can view a factor graph as a bipartite graph with node sets $V_G$ and $F_G$ such that each
$a\in F_G$ is adjacent to the variable nodes $\partial_{G}(a,j)$ for $j\leq k_{\psi_a}$.
Hence, we call them the {\em neighbors} of $a$ and take license to just write $\partial_Ga$ for the set of neighbors.
Conversely, we write $\partial_Gx$ for the set of constraint nodes $a$ such that $x\in\partial_Ga$.
By {\bf GR3} $\partial_Gx$ is a finite set for every $x$.
Further, a {\em rooted factor graph} is a connected factor graph  together with a distinguished variable node $r$, its {\em root}.

However, we keep in mind that the ``bipartite graph'' point of view looses a bit of information.
Indeed, the neighbors of $a$ are {\em ordered}.
This is important in evaluating (\ref{eqGibbs1}) because the $\psi_a$ need not be permutation invariant.

\begin{example}[The Ising model on the grid $\ZZ^2$]
Let $G_n$ be a finite subgraph of $\ZZ^2$ (say a large box), $\Omega = \{ \pm 1\}$, and $\psi: \{ \pm 1 \}^2 \to (0,\infty)$ be defined by $\psi(x_1,x_2) = \exp ( \beta \cdot x_1x_2)$ for a fixed real number $\beta$.  Then the Ising model on $G_n$ is defined by
\begin{align*}
\mu_{G_n} (\sigma) &= \frac{\prod_{(u,v) \in E (G_n)} \psi (\sigma(u), \sigma(v) ) }{Z_{G_n}}\,,
\end{align*}
where
\begin{align*}
Z_{G_n} &= \sum_{\sigma \in \{ \pm 1 \}^{V(G)}} \prod_{(u,v) \in E (G_n)} \psi (\sigma(u), \sigma(v) )  \,.
\end{align*}
 If $\beta >0$ we say the model is \textit{ferromagnetic} (like spins preferred across edges) and if $\beta <0$, the model is \textit{antiferromagnetic}.  
\end{example}

\begin{example} [The (positive-temperature) $k$-SAT model]\label{Ex_kSAT}
The $k$-SAT model is an example of a factor model with multiple constraint types.  We take $\Omega = \{ \pm 1 \}$ and have    $2^k$ constraint types, each of arity $k$, indexed by vectors  $\vec c \in \{ \pm 1 \}^k$, with
\begin{align*}
\psi^{(\vec c)}(\vec x) &= \exp ( - \beta \mathbf 1\cbc{\vec c \cdot \vec x = -k} ) 
\end{align*}
 for $\vec x \in \{\pm 1 \}^k$. The temperature parameter  $\beta$ is a fixed real number controlling how much satisfied clauses are preferred to unsatisfied clause.  The $2^k$ different constraint function types correspond to the $2^k$ different ways to assign signs to $k$ variables that appear together in a $k$-CNF clause.  We have $\vec c \cdot \vec x = - k$ if all $k$ signed variables are $-1$, in which case the clause is unsatisfied. We form a random instance of the $k$-SAT model by choosing a random number of constraints of each of the $2^k$ types to form $F_G$, and to each constraint $a \in F_G$ attaching a uniformly random ordered set of $k$ variable nodes $x_{a_1}, \dots x_{a_k}$  from $V_G = \{ x_1, \dots x_n \}$. The Gibbs measure is then a probability distribution over all assignments to the $n$ variables  given by
 \begin{align*}
 \mu_{G} (\sigma)&= \frac{ \prod_{a \in F_G} \psi_a(\sigma(x_{a_1}), \dots \sigma(x_{a_k}))}{Z_G}  \, 
 \qquad\mbox{where}\qquad
 	Z_G=\sum_\sigma\prod_{a \in F_G} \psi_a(\sigma(x_{a_1}), \dots \sigma(x_{a_k})).
\end{align*}
Studying $\mu_G$ can be viewed as a generalization of the MAX $k$-SAT problem~\cite{maxsat}.
In fact, the maximum number of clauses that can be satisfied simultaneously comes out as $|F_G|+\lim_{\beta\to\infty}\frac{\partial}{\partial\beta}\ln Z_G$.
\end{example}  

Apart from natural geometric factor graph models such as the Ising model,
there is substantial interest in models where the geometry of interactions is random like in Example~\ref{Ex_kSAT}.
For instance, such models appear in the statistical mechanics of disordered systems,  coding theory and, of course,
the theory of random graphs itself~\cite{JLR,MM,Rudi}.
Perhaps the simplest and most natural way of defining such models is by extension of the \Erdos-\Renyi\ model.
Hence, given a set $\Psi$ of possible weight functions and a sequence $\rho=(\rho_\psi)_{\psi\in\Psi}$ of positive reals we define
the random factor graph $\G_n=\G_n(\Psi,\rho)$ as follows.
The set of variable nodes is $V_n=\{x_1,\ldots,x_n\}$.
Moreover, choosing $m_\psi=\Po(n\rho_\psi)$ independently for each $\psi\in\Psi$, we define the set of constraint nodes
as $$F_n=\{a_{\psi,i}:i\leq m_\psi\mbox{ for all }\psi\in\Psi\}.$$
Further, independently for each $a_{\psi,i}\in F_n$ choose $\partial_{\G_n} a_{\psi,i}\in V_n^{k_\psi}$ uniformly at random.

A random factor graph $\G_n$ induces a Gibbs measure $\mu_{\G_n}$ on $\Omega^n$ via (\ref{eqGibbs1}).
Thus, we could just apply the ``limit theory'' for discrete measures from \Sec~\ref{Sec_cubes} to the sequence $(\mu_{\G_n})_n$.
Indeed, this is essentially what Panchenko~\cite{Panchenko} {does} (using the Aldous-Hoover representation instead of \Sec~\ref{Sec_cubes}).
However, the geometry of the sparse random graph $\G_n$ contains additional information that is not directly encoded in the measure $\mu_{\G_n}$.
Hence, the basic idea in the following is to study the convergence of the sequence $(\mu_{\G_n})_n$ of measures
jointly with the convergence of the geometry of the factor graph $\G_n$.
This will enable us to identify the limit of $(\mu_{\G_n})_n$ with a ``geometric'' measure on a (possible infinite) random tree.
In particular, we aim to use these insights to get a handle on the {\em free energy} of the model, defined as
	\begin{equation}\label{eqFreeEnergy}
	\lim_{n\to\infty}\frac1n\Erw[\ln Z_{\G_n}],
	\end{equation}
provided that the limit even exists.
(Of course, $\Psi,\rho$ remain fixed as $n$ grows.)

A similar ``geometric'' approach was pursued in~\cite{Victor} for a more general class of factor graph models, and without a proper notion
of convergence of the sequence of Gibbs measures.
For the more specific models studied here we will obtain somewhat stronger results from simpler proofs.

\subsection{Local weak convergence}
To carry out this program we need to set up a notion of a ``limit'' of the geometry of $\G_n$ as $n\to\infty$.
To adapt the appropriate formalism of ``local weak convergence''~\cite{AldousLyons,BordenaveCaputo,Lovasz}
to our context, let $G=(V_G,F_G,\partial_G,(\psi_a)_{a\in F_G},r)$, $G'=(V_{G'},F_{G'},\partial_{G'},(\psi_a)_{a\in F_G},r')$ be two rooted factor graphs.
An {\em isomorphism} $f:G\to G'$ is a bijection $f:V_T\cup F_T\to V_{T'}\cup F_{T'}$ such that
\begin{description}
\item[ISM1] $f(V_T)=V_{T'}$, $f(F_T)=F_{T'}$, $f(r)=r'$,
\item[ISM2] $\psi_{f(a)}=\psi_a$ for all $a\in F_T$,
\item[ISM3] $\partial_{T'}(f(a),j)=f(\partial_{T}(a,j))$ for all $j\in[k_{\psi_a}]$.
\end{description}
We write $G\ism G'$ if there is an isomorphism $G\to G'$.

Let $[G]$ be the isomorphism class of $G$ and let $\fG$ be the set of all isomorphism classes.
Further, for an integer $\ell\geq1$ let $\partial^\ell G$ be obtained from $G$ by deleting all (variable and constraint) nodes whose distance from the root exceeds $2\ell$.
Then it makes sense to write $\partial^\ell[G]$, because $G\ism H$ implies that $\partial^\ell G\ism\partial^\ell H$ for all $\ell$.

To be allowed to use standard graph terminology for isomorphism classes, let us pick one representative $G_0$ of every isomorphism class $[G]$ arbitrarily.
Hence, if we speak, e.g., of ``the neighbor of the root of $[G]$'', we refer to the corresponding object in the chosen representative $G_0$.

We endow $\fG$ with the coarsest topology that makes all the functions
	\begin{align}\label{eqLWC}
	\fG\to\{0,1\},\qquad [G]\mapsto\vecone\{\partial^\ell G\ism\partial^\ell H\}\qquad ([H]\in\fG)
	\end{align}
continuous.
Moreover, let $\fT\subset\fG$ be the set of all isomorphism classes of all acyclic rooted factor graphs with the induced topology.
The spaces $\fG,\fT$ are Polish~\cite{AldousLyons,BordenaveCaputo}.
Hence, so are the spaces $\cP(\fG)$, $\cP(\fT)$ of probability measures on $\fG,\fT$ with the weak topology.
Additionally, we equip the spaces $\cP^2(\fG)$, $\cP^2(\fT)$ of probability measures on $\cP(\fG)$, $\cP(\fT)$ with the weak topology as well.

For a factor graph $G$ and a variable node $v$ let $G_{\reroot v}$ be the connected component of $v$ in $G$ rooted at $v$.
Thus, $G_{\reroot v}$ is a rooted factor graph.
Similarly, if $(G,r)$ is a rooted factor graph, then $G_{\reroot v}=(G,v)$ is obtained by re-rooting at $v$.
Each factor graph $G$ induces an empirical distribution on $\fG$, namely
	\begin{align*}
	\Lambda_G&=|V_G|^{-1}\sum_{x\in V_G}\atom_{G_{\reroot x}}\in\cP(\fG).
	\end{align*}
Hence, the random factor graph gives rise to a distribution
	\begin{align*}
	\Lambda_n&=\Erw[\atom_{\Lambda_{\G_n}}]\in\cP^2(\fG).
	\end{align*}
Due to the definition (\ref{eqLWC}) of the topology, this measure captures the distribution of the ``local structure'' of $\G_n$,
i.e., the ``statistics'' of the bounded-size neighborhoods.

Guided by the example of the \Erdos-\Renyi\ random graph, we expect that the local structure of $\G_n$ is described by a branching process.
Specifically,
starting from a single variable node $\cV_0=\{x_0\}$ and with $T_0$ the tree consisting of $x_0$ only, we build a sequence of random trees $(\T_\ell)_\ell$ as follows.
Let $(A_{\psi,x})_{\psi\in\Psi,x\in\cV_\ell}$ be a family of independent random variables such that $A_{\psi,x}$ has distribution $\Po(\rho_{\psi})$.

Now, obtain $\T_{\ell+1}$ from $\T_{\ell}$ by attaching $Y_{\psi,x}$ children, which are constraint nodes with weight function $\psi$, to each $x\in\cV_\ell$.
For each of them independently choose the position that the parent variable occupies uniformly and independently from $[k_\psi]$
and attach $k_{\psi}-1$ further variable nodes.
Finally, let $\cV_{\ell+1}$ be the set of variable nodes of $\T_{\ell+1}$ at distance precisely $2(\ell+1)$ from the root.

Let $\thet_\ell\in\cP(\fT_\ell)$ be the distribution of the random tree $\T_\ell$.
Because the topology is generated by the functions (\ref{eqLWC}), $(\thet_\ell)_{\ell\geq1}$ is a Cauchy sequence.
Since $\cP(\fT)$ is a complete space, there exists a limit $\thet\in\cP(\fT)$.
Further, write $\T$ for a random (possibly infinite) tree drawn from $\thet$.

\begin{proposition}\label{Prop_lwc}
We have $\lim_{n\to\infty}\Lambda_n=\atom_{\thet}$.
\end{proposition}
\begin{proof}
For $T\in\fT$ and $\ell\geq1$ let $Q_{T,\ell}(\G_n)$ be the number of variable nodes $x$ of $\G_n$ such that
$\partial^\ell\G_{n\reroot x}\ism\partial^\ell T$.
Unravelling the construction of the topology via (\ref{eqLWC}), we see that $\lim_{n\to\infty}\Lambda_n=\atom_{\thet}$ iff
$n^{-1}Q_{T,\ell}(\G_n)$ converges in probability to $\pr\brk{\partial^\ell\T=\partial^\ell T}$ for every $T,\ell$.

Hence, fix $T,\ell$ and assume that $n$ is large.
Further, let $\G_n'$ be a random factor graph with variable nodes $V_n$ in which for every $\psi\in\Psi$
each of the $n^{k_\psi}$ possible constraint with weight function $\psi$ is present with probability $p_\psi=\rho_\psi n^{1-k_\psi}$ independently.
Then the number $M_\psi$ of constraints of type $\psi$ has distribution $\Bin(n^{k_\psi},p_\psi)$, and they are mutually independent.
Because the total variation distance of $M_\psi$ and $\Po(n\rho_\psi)$ is $o(1)$ as $n\to\infty$, the same is true of the
random graph distributions $\G_n,\G_n'$.

We are now going to show by induction on $\ell$ that there is a coupling of
$[\partial^\ell_{\G_n}x_1]$ and $\partial^\ell\T$ such that both coincide with probability $1-o(1)$.
For $\ell=0$ there is nothing to show as both graphs consist of the root only.
To proceed from $\ell$ to $\ell+1$, let $W_\ell$ be the set of variable nodes at distance precisely $2\ell$ from $x_1$ in $\G_n'$
and let $V_\ell$ be the set of variable nodes at distance precisely $2\ell$ from $r_{\T}$.
Further, condition on the event $\cE_\ell=\{[\partial^\ell_{\G_n}x_1]=\partial^\ell\T\}$ and fix an isomorphism
	$$\varphi:\partial^\ell_{\G_n}x_1\to\partial^\ell\T.$$
Moreover, let $\cX$ be the event that the random factor graph $\G_n$ either contains a constraint node $a$
such that $|\partial_{\G_n}a\cap W_\ell|\geq2$ or a constraint nodes $b,c$ such that 
$|\partial_{\G_n}b\cap W_\ell|,|\partial_{\G_n}c\cap W_\ell|=1$ and $\partial_{\G_n}b\cap\partial_{\G_n}c\setminus W_\ell\neq\emptyset$.
Because the tree $T$ remains fixed as we let $n\to\infty$, the number of possible $a,b,c$ with these properties is $O(n^{k_\psi-2})$ for every $\psi$.
Therefore, $\pr\brk{\cX|\cE_\ell}=O(1/n)$.
Furthermore, for every $x\in W_\ell$ let number $D_{x,\psi}$ be the number of constraint nodes $a$ of type $\psi$ such that
$\partial_{\G_n}a\cap\bigcup_{l\leq\ell} W_l=\{x\}$.
Given $\cE_\ell\cap\cX$, $D_{x,\psi}$ has distribution $k_\psi\cdot\Bin((n-|\bigcup_{l\leq\ell} W_l|-O(1))^{k_\psi-1},q_\psi)$,
and the $D_{x,\psi}$ are asymptotically independent.
Analogously, let $d_{x,\psi}$ be the number of constraint nodes $a$ of type $\psi$ that are children $x\in V_\ell$.
Then $d_{x,\psi}$ has distribution $\Po(k_\psi\rho_\psi)$ by the construction of $\T$, and the $d_{x,\psi}$ are mutually independent.
Consequently, since $|W_\ell|=O(1)$ as $n\to\infty$, there exists a coupling of the vectors $(d_{x,\psi})_{x\in W_\ell,\psi\in\Psi}$,
$(D_{x,\psi})_{x\in V_\ell,\psi\in\Psi}$ such that $\pr\brk{\forall x,\psi:d_{x,\psi}=D_{\varphi(x),\psi}|\cX\cap\cE_\ell}=1-o(1)$.
Hence, we obtain a coupling of $\partial^{\ell+1}_{\G_n}x_1$ and $\partial^{\ell+1}\T$ such that
$\pr[[\partial^{\ell+1}_{\G_n}x_1]=\partial^{\ell+1}\T]=1-o(1)$, as desired.

Because the random factor graph model is invariant under permutations of the variable nodes,
the existence of this coupling implies that $\Erw[Q_{T,\ell}(\G_n)]\sim n\pr\brk{\partial^\ell\T=\partial^\ell T}$.
To estimate the second moment $\Erw[Q_{T,\ell}(\G_n)^2]$, we 
repeat the argument from the previous paragraph for the two variable nodes $x_1,x_2$ to show that
	\begin{equation}\label{eqsmm}
	\pr\brk{\partial^\ell_{\G_n}x_1=T,\partial^\ell_{\G_n}x_2=T'}\sim\pr\brk{\partial^\ell\T=\partial^\ell T}\pr\brk{\partial^\ell\T=\partial^\ell T}.
	\end{equation}
Once more by permutation-invariance, (\ref{eqsmm}) implies that $\Erw[Q_{T,\ell}(\G_n)^2]\sim\Erw[Q_{T,\ell}(\G_n)]^2$.
Finally, the desired convergence in probability follows from Chebyshev's inequality.
\end{proof}

\subsection{Replica symmetry}
Having discussed the meaning of ``convergence of the local structure'', let us now return to the convergence of the Gibbs measure $\mu_{\G_n}$ itself.
The ``cavity method'', a non-rigorous but sophisticated approach from statistical physics~\cite{pnas,MM},
predicts a {relatively} simple formula for the free energy~(\ref{eqFreeEnergy})
if $\mu_{\G_n}$ converges to an atom $\atom_w$, $w\in\step_\Omega$, in the metric $\cutm(\nix,\nix)$. 
This convergence assumption roughly coincides with the {\em replica symmetry} condition from physics~\cite{pnas,Panchenko}.
According to the cavity method, in the replica symmetric case the free energy can be calculated by applying an explicit functional, the {\em Bethe free energy}~\cite{YFW},
to a fixed point of a message passing scheme called {\em Belief Propagation}~\cite{MM}.
We are going to vindicate this prediction.

But before we introduce Belief Propagation and the Bethe free energy, let us briefly discuss the replica symmetry assumption.
Formally, we are going to assume that there is a function $w\in\step_\Omega$ such that
	\begin{equation}\label{eqRS}
	\lim_{n\to\infty}\Erw\brk{\cutm(\mu_{\G_n},{\atom_w})}=0.
	 \end{equation}
In other words, $\mu_{\G_n}$ converges to $\atom_w$ in probability with respect to $\cutm(\nix,\nix)$.

The assumption (\ref{eqRS}) holds in all examples of random factor graph models where we currently have an at least somewhat explicit formula for the free energy
	(to our knowledge).
For example, this includes all cases in which the free energy can be computed by the ``second moment method'' (e.g., \cite{nae,ANP,Nor}).
Indeed,  in these examples $w:x\in[0,1)\mapsto p\in\cP(\Omega)$ is a {\em constant} function.
However, there are replica symmetric models in which the limiting density $w$ is not constant.

Instead of relying on the second moment method, the condition (\ref{eqRS}) can be checked (and the function $w$ can be computed)
by studying spatial mixing properties of the Gibbs measure; for an example see~\cite{Victor2}.
Let us give a simple generic proof that the {\em non-reconstruction condition},  a spatial mixing property, entails~(\ref{eqRS});
this was predicted in~\cite{pnas}.

For a random factor graph $\G_n$, a variable node $x\in V_n$ and $\ell\geq1$ let
$\nabla_\ell(\G_n,x)$ be the $\sigma$-algebra on $\Omega^{V_n}$ generated by the events
	$\{\SIGMA(y)=\omega\}$ for all $\omega\in\Omega$ and all variables $y$ at distance greater than $2\ell$ from $x$.
Thus, in the measure $\mu_{\G_n}[\nix|\nabla_\ell(G,x)]$ we condition on all the values of all the variable nodes at distance greater than $2\ell$ from $x$.
The random factor graph model has the {\em non-reconstruction property} if
	\begin{align}\label{eqNonReconstruction}
	\lim_{\ell\to\infty}\lim_{n\to\infty}
			\Erw\bck{\TV{\mu_{\marg x_1}-\mu_{\marg x_1}\brk{\nix|\nabla_\ell(\G_n,x_1)}}}_{\mu_{\G_n}}=0.
	\end{align}
To parse (\ref{eqNonReconstruction}), we note that the outer expectation $\Erw[\nix]$ refers to the choice of the random factor graph $\G_n$.
Further, the outer mean $\bck{\nix}_{\G_n}$ over the Gibbs measure of $\G_n$ generates the random boundary condition.
We then compare the conditional marginal  $\mu_{\marg x_1}\brk{\nix|\nabla_\ell(\G,_nx_1)}$ 
given the boundary condition with the unconditional marginal $\mu_{\marg x_1}$.
Because the distribution of $\G_n$ is invariant under permutations of the variables, the choice of the variable $x_1$ in (\ref{eqNonReconstruction})
is irrelevant.
Hence, (\ref{eqNonReconstruction}) provides that the impact of a random boundary condition on the marginal of
any specific variable $x_i$ diminishes in the limit $\ell,n\to\infty$.

\begin{proposition}\label{Thm_nonre}
Assume that $\lim_{n\to\infty}\Erw[\cutm(\mu_{\G_n},\mu)]=0$ for some $\mu\in M_\Omega$.
If (\ref{eqNonReconstruction}) holds,
then there exists $w\in\step_\Omega$ such that $\bar\mu=\overline{\atom_w}$; thus, (\ref{eqRS}) holds.
\end{proposition}
\begin{proof}
We apply an argument from~\cite{Mossel} developed for the ``stochastic block model'' to our setup.
Due to \Cor~\ref{Cor_factorise} and because the distribution of $\G_n$ is invariant under permutations of the variables it suffices to prove 
	\begin{align}\label{eqThm_nonre1}
	\lim_{n\to\infty}\Erw\TV{\mu_{\G_n\marg x_1,x_2}-\mu_{\G_n\marg x_1}\tensor\mu_{\G_n\marg x_2}}=0.
	\end{align}
Hence, assume that (\ref{eqNonReconstruction}) holds but (\ref{eqThm_nonre1}) does not.
Then there exist $\omega_1,\omega_2\in\Omega$ and $0<\eps<0.1$ such that for infinitely many $n$ we have
	\begin{align}\label{eqThm_nonre2}
	\pr\brk{\abs{\bck{\vecone\{\SIGMA_{x_1}=\omega_1\}|\SIGMA_{x_2}=\omega_2}_{\mu_{\G_n}}-
		\bck{\vecone\{\SIGMA_{x_1}=\omega_1\}}_{\mu_{\G_n}}
		}>2\eps,\ \bck{\vecone\{\SIGMA_{x_2}=\omega_2\}}_{\mu_{\G_n}}>2\eps}>2\eps.
	\end{align}
Thus, let $\ell$ be a large enough integer and let $\cE$ be the event that the distance between $x_1,x_2$ in $\G_n$ is greater than $2\ell$.
Because our factor graph $\G_n$ is sparse and random, we have $\pr\brk{\cE}=1-o(1)$ as $n\to\infty$.
Therefore, (\ref{eqThm_nonre2}) implies
	\begin{align}\label{eqThm_nonre3}
	\pr\brk{\abs{\bck{\vecone\{\SIGMA_{x_1}=\omega_1\}|\SIGMA_{x_2}=\omega_2}_{\mu_{\G_n}}-
		\bck{\vecone\{\SIGMA_{x_1}=\omega_1\}}_{\mu_{\G_n}}}>\eps,\ 
		\bck{\vecone\{\SIGMA_{x_2}=\omega_2\}}>\eps,
		\ \cE}>\eps.
	\end{align}
To complete the proof, let $\cS$ be the set of all $\sigma\in\Omega^{V_n}$ such that $\sigma_{x_2}=\omega_2$.
If the event $\cE$ occurs, then given $\nabla_\ell(\G_n,x_1)$ the value assigned to $x_2$ is fixed.
Therefore, (\ref{eqThm_nonre3}) implies
	\begin{align*}
	\Erw\bck{\TV{\mu_{\marg x_1}-\mu_{\marg x_1}\brk{\nix|\nabla_\ell(\G_n,x_1)}}}_{\mu_{\G_n}}
		&\geq\Erw\brk{\vecone\{\cE\}\bck{\TV{\mu_{\marg x_1}-\mu_{\marg x_1}\brk{\nix|\nabla_\ell(\G_n,x_1)}}|\cS}_{\mu_{\G_n}}
			\bck{\vecone\{\SIGMA\in\cS\}}_{\mu_{\G_n}}}\geq\eps^3,
	\end{align*}
in contradiction to (\ref{eqNonReconstruction}).
\end{proof}

Assumption (\ref{eqNonReconstruction}), and hence (\ref{eqRS}),
is a weaker than the assumption of {\em Gibbs uniqueness}, another spatial mixing property.
Under this stronger assumption Dembo, Montanari, and Sun \cite{dembo} used an interpolation scheme to compute the free energy in a wide variety 
of factor models on graphs converging locally to random trees.
Further related work on Gibbs uniqueness and/or the interpolation method includes~\cite{bayati,David2,David}.

Although non-reconstruction is sufficient for (\ref{eqRS}) to hold, it is not a necessary condition.
Yet there are quite a few examples of random factor graph models where the condition is expected (or known)
	to hold but where the free energy has not been computed rigorously
	(e.g., the random $k$-SAT model~\cite{pnas} or planted models~\cite{FlorentLenka}).
We expect that these could be tackled via \Prop~\ref{Thm_nonre} and the other results in this section.

\subsection{Belief Propagation}
To establish a connection between the Gibbs measure $\mu_{\G_n}$ and the limit $\thet$ of the local structure of the factor graph
we are going to use the Belief Propagation scheme, which plays a key role in the physicists' ``cavity method''~\cite[\Chap~14]{MM}.
In fact, due to the Poisson structure of the tree distribution $\thet$, Belief Propagation takes a relatively simple form in our setting.

If we look at the random graph $\G_n$, then by \Prop~\ref{Prop_lwc} $\thet$ gives the fraction of variable nodes $x_i$
such that $\partial^\ell_{\G_n}T\ism\partial^\ell T$ for every tree $T$.
Suppose that for each tree $T$ we record the empirical distribution of the marginals $\mu_{\G_n\marg x_i}$ of such variable nodes.
Since each marginal is a distribution on $\cP(\Omega)$, this empirical distribution lies in $\cP^2(\Omega)=\cP(\cP(\Omega))$.
Hence, we obtain a map $\fT\to\cP^2(\Omega)$.
According to the Belief Propagation equations, this map must satisfy a certain ``consistency condition''.
To be specific,  for a variable node $v$ of $T$ rooted at $r_T$ let $\partial_{T\downarrow}v$ be the set of all children of $v$.
Moreover, let $T_{\downarrow v}$ be the tree ``pending on $v$'', i.e., the connected component of $v$
in the tree obtained from $T$ by removing the neighbor $a$ on the path from $v$ to $r_T$.
Then the marginal distribution of $T$ must be ``consistent'' with the marginal distributions of the trees $T_{\downarrow v}$
for $v$ at distance exactly two from the root $r_T$.

To formalize this, 
we call a measurable map $\nu^\star:\fT\to\cP^2(\Omega)$, $T\mapsto\nu_{T}^\star$ a {\em $\thet$-Belief Propagation fixed point} 
if the following condition holds for $\thet$-almost all trees $T\in\fT$.
Independently for each variable $y$ at distance precisely two from $r_T$ choose
$\eta_{T,y}\in\cP(\Omega)$ from the distribution $\nu_{T_{\downarrow y}}^\star$.
Moreover, for each constraint node $a\in\partial_Tr_T$ let
	\begin{align}
		\hat\eta_{T,a}(\omega)&\propto
		{\sum_{\sigma\in\Omega^{\partial_Ta}}
			\vecone\{\sigma_{r_T}=\omega\}\psi_a(\sigma)\prod_{y\in\partial_{T\downarrow}a}\eta_{T,y}(\sigma_y)}
			&(\omega\in\Omega).
			\label{eqmsg2}
	\end{align}
Then we require that 
	\begin{align}			\label{eqmsg1}
	\eta_{r_T}(\omega)&\propto{\prod_{a\in\partial_T r_T}\hat\eta_{T,a}(\omega)}&(\omega\in\Omega)
	\end{align}
has distribution  $\nu_{T}^\star$.
The idea behind (\ref{eqmsg2})--(\ref{eqmsg1}) is that the 
marginal distribution of the spin of the root variable behaves as though the
the spins assigned to the roots of the subtrees were independent.
For a detailed derivation of the Belief Propagation equations see~\cite[\Chap~14]{MM}.

We would like to show that under the assumption (\ref{eqRS}) the marginals of the Gibbs measure $\mu_{\G_n}$ ``converge to'' a Belief Propagation fixed point.
To this end, we define for a tree $T\in\fT$, an integer $\ell\geq0$ and a factor graph $\G_n$ the distribution
	$\nu_{\G_n,T,\ell}\in\cP(\Omega)^2$ as follows.
Let $V(\G_n,T,\ell)$ be the set of all $x_i\in V_n$ such that $\partial_{\G_n}^\ell x_i\ism\partial^\ell T$.
If $V(\G_n,T,\ell)\neq\emptyset$ we let
	$$\nu_{\G_n,T,\ell}=\frac1{|V(\G_n,T,\ell)|}\sum_{x\in V(\G_m,T,\ell)}\atom_{\mu_{\G_n\marg x}}.$$
Thus, $\nu_{\G_n,T,\ell}$ is the empirical distribution of the Gibbs marginals $\mu_{\G_n\marg x}$ for $x\in V(\G_n,T,\ell)$.
If $V(\G_n,T,\ell)=\emptyset$, we let $\nu_{\G_n,T,\ell}$ be the uniform distribution on $\cP(\Omega)$, say.
Recall that $d_1(\nix,\nix)$ denotes the $L_1$-Wasserstein metric.

\begin{theorem}\label{Thm_BP}
If (\ref{eqRS}) holds, then there exists a $\thet$-Belief Propagation fixed point $\nu^\star_{\T}$ such that
	\begin{equation}
	\label{eqThm_BP}
	\lim_{\ell\to\infty}\lim_{n\to\infty}\Erw_{\G_n,\T}[d_1(\nu_{\G_n,\T,\ell},\nu_{\T}^\star)]=0.
	\end{equation}
\end{theorem}

Thus, for large enough $\ell,n$ and for a random tree $\T$ the empirical distribution of the marginals
of those variables of $\G_n$ whose depth-$\ell$ neighborhood is isomorphic to $\partial^\ell\T$ is close to $\nu_{\T}^\star$.
In particular, in the limit $n\to\infty$
the Gibbs measures $\mu_{\G_n}$ on the random factor graph induce a Belief Propagation fixed point on the limiting random tree $\T$.

Panchenko~\cite{Panchenko} proved a related fixed point property under different assumptions.
While he does not require the convergence in probability assumption~(\ref{eqRS}), he imposes certain additional conditions
on the random factor graph model (which may not be strictly necessary).
Moreover, Panchenko's result just provides a ``single-level'' distributional fixed point equation,
rather than a full ``unwrapping'' into a $\thet$-Belief Propagation fixed point.
The proof of \Thm~\ref{Thm_BP} is by extension of the argument from~\cite[\Sec~2]{Panchenko};
in fact, the ``unwrapping'' requires a fair amount of work.

The proof of \Thm~\ref{Thm_BP} yields the following ``geometric'' fact about the interplay of the marginals on the random factor graph.
For a factor graph $G$ and a variable node $x$ let $G-x$ be the factor graph obtained by removing $x$ and its adjacent constraint nodes.

\begin{corollary}
\label{cor:marginalneighbors}
If (\ref{eqRS}) holds then
\begin{align*}
\lim_{n\to\infty}\frac1n\sum_{x\in V_n}\sum_{\omega\in\Omega}\Erw\abs{
	\mu_{\G_n\marg x} (\omega)-
	\frac{   \prod_{a\in\partial_{\G_n} x} \sum_{\vec s:\partial_{\G_n}a\to\Omega} \mathbf 1\cbc{s_x=\omega} \psi_a (\vec s) 
		\prod_{y\in\partial_{\G_n-x}a}\mu_{\G_n{-x}\marg y}(\vec s_y) }
			{  \sum_{\omega^\prime \in \Omega}   \prod_{a\in\partial_{\G_n} x} \sum_{\vec s:\partial_{\G_n}a\to\Omega}
				\mathbf 1\cbc{s_x=\omega^\prime} \psi_a (\vec s) \prod_{y\in\partial_{\G_n-x}a}\mu_{\G_n{-x}\marg y}(\vec s_y)}
	}=0,
\end{align*}
where the expectation is over the choice of the random factor graph $\G_n$. 
\end{corollary}

\subsection{The Bethe free energy}\label{Sec_BetheFreeEnergy}

According to the ``cavity method'', we can extract $\frac{1}{n}\Erw[\ln Z_{\G_n}]$ from the Belief Propagation fixed point from \Thm~\ref{Thm_BP}
via a formula called the  Bethe free energy~\cite[\Sec~14.2.4]{MM}.
Suppose that $\nu^\star$ is a Belief Propagation fixed point.
For a tree $T\in\fT$ with root $r_T$ and 
$y$ at distance two from $r_T$ let
$\eta_{T,y}\in\cP(\Omega)$ be independently chosen from $\nu_{T_{\downarrow y}}^\star$.
Moreover, for $a\in\partial_T r_T$ define  $\hat\eta_{T,a}\in\cP(\Omega)$ as in (\ref{eqmsg2}).
Further, let
	\begin{align}\label{eqDefBFE1}
	\tilde\eta_{T,a}(\sigma)&\propto{\prod_{b\in\partial_T r_T\setminus\{a\}}\hat\nu_{T,b}(\sigma)},&
		&(\sigma\in\Omega),\\
	\varphi_{T}&=\ln\sum_{\sigma\in\Omega}\prod_{a\in\partial_T r_T}\hat{\eta}_{T,a}(\sigma),&
	\hat\varphi_{T,a}&=\ln\sum_{\sigma\in\Omega^{\partial_Ta}}
		\psi_a(\sigma)\tilde\eta_{T,a}(\sigma_{r_T})\prod_{y\in\partial_{T\downarrow}a}\eta_{T,y}(\sigma_y),\label{eqDefBFE2}\\
	\tilde\varphi_{T,a}&=\ln\sum_{\sigma\in\Omega}\tilde\eta_{T,a}(\sigma)\hat\eta_{T,a}(\sigma).
		\label{eqDefBFE3}
	\end{align}
(The arguments of all the above logarithms are strictly positive.
Indeed, because the functions $\psi\in\Psi$ take strictly positive values, (\ref{eqmsg2}) ensures that $\hat\eta_{T,a}(\omega)>0$ for all $\omega\in\Omega$.
Hence, $\varphi$ is well-defined.
So are $\hat\varphi_a$, $\tilde\varphi_a$, by the same token.)
Taking the expectation over $\T$ and independent $\eta_y\in\cP(\Omega)$ for all $y$ at distance two from $r_T$, we define
the {\em Bethe free energy} as
	\begin{align}\label{eqDefBFE}
	\cB_\thet(\nu_{\T}^\star)&=
		\Erw\brk{\varphi_{\T}+\sum_{a\in\partial_{\T}r_{\T}}\bc{\frac{\hat\varphi_{\T,a}}{k_a}
			-\tilde\varphi_{\T,a}}}.
	\end{align}

\begin{theorem}\label{Thm_Bethe}
Assume that (\ref{eqRS}) holds and let $\nu_{\T}^\star$ be a $\thet$-Belief Propagation fixed point such that (\ref{eqThm_BP}) holds.
Then 
	$$\lim_{n\to\infty}\frac1n\Erw[\ln Z_{\G_n}]=\cB_\thet(\nu^\star_{\T}).$$
\end{theorem}

Hence, \Thm s~\ref{Thm_BP} and~\ref{Thm_Bethe} show that under the assumption (\ref{eqRS}) the free energy of the random
factor graph model comes out in terms of a ``geometric'' measure, i.e., a Belief Propagation fixed point on the limiting tree
that captures the local structure of the factor graph.

The proof of \Thm~\ref{Thm_Bethe} is by adapting ideas from Panchenko~\cite[\Sec~2]{Panchenko} to our situation.
In particular, we combine the convergence of the measure $\mu_{\G_n}$ with
a technique from Aizenman, Simms, Starr~\cite{Aizenman}.
The difference between \Thm~\ref{Thm_Bethe} and \cite{Panchenko} is that the latter requires
certain conditions on the weight functions $\psi\in\Psi$ (to facilitate an interpolation argument) but does not require (\ref{eqRS}).
(However, it is stated without proof in~\cite{Panchenko} that the free energy can be derived along the lines of that paper under the assumption
(\ref{eqRS}).)

\subsection{Proof of \Thm~\ref{Thm_BP}}\label{Sec_BP}
We begin by constructing  a family of $\cP(\Omega)$-valued random variables $(X_\ell)_\ell$ as follows.
Let $\vec\tau=(\vec\tau_i)_{i\geq1}$ be an i.i.d.\ family of $\cP(\Omega)$-valued random variables with distribution $w_{\vec x}$ for a
uniformly random $\vec x\in[0,1)$.
Then $X_\ell=X_\ell(\T,\vec\tau)$ is defined as follows.
We simply set $X_0(\T,\vec\tau)=\vec\tau_1$.
Further, to define $X_\ell(\T,\vec\tau)$ for  $\ell\geq1$ let $V_{\ell}=\{x_{\ell,1},\ldots,x_{\ell,L}\}$ be the set of variable nodes at distance
precisely $2\ell$ from the root $r_{\T}$, ordered in some arbitrary but deterministic way.
Then we let  $X_\ell(\T,\vec\tau)$ be the distribution of the spin $\SIGMA_{r_{\T}}$ under
the Gibbs measure of the tree $\partial^\ell\T$ with a boundary condition chosen independently from $\TAU_1,\ldots,\TAU_L$.
In symbols,
	\begin{align*}
	X_\ell(\T,\vec\tau)&=\sum_{\sigma\in\Omega^L}
		\brk{\atom_{\bck{\SIGMA_{r_{\T}}(\nix)|\SIGMA(x_{\ell,1})=\sigma_1,\ldots,\SIGMA(x_{\ell,L})=\sigma_L}_{\mu_{\partial^\ell\T}}}
			\prod_{i\leq L}\TAU_i(\sigma_i)}
			\in\cP(\Omega).
	\end{align*}

Now we give an alternative construction.
Again set $X_0=\vec\tau_1$.
Now for $\ell \ge 1$, let $X_\ell(\T,\vec\tau)$ be the distribution of the spin at the root given that its neighbors have
 distribution $X_{\ell-1}(\T_1, \vec \tau^{1}), X_{\ell-1}(\T_2, \vec \tau^{2}), \dots$ 
 where $\T_1$ is the tree appending the $i$th neighbor of the root of $\T$, and 
 $\vec \tau^1, \vec \tau^2, \dots$ are independent copies of $\vec\tau$.  Observe that the two constructions do in fact give the same distribution.    

The main step of the proof is to compare $X_\ell$ with the empirical distribution on the random factor graph.
Recall that $d_1(\nix,\nix)$ denotes the Wasserstein metric.

\begin{lemma}\label{Lemma_WillsCoupling}
For every $T\in\fT$ the following is true.
Let $Y_\ell(T,\G_n)$ be the empirical distribution of the marginals of the variables $x$ such that $\partial^\ell_{\G_n}x\ism\partial^\ell T$.
Then
	\begin{equation}\label{eq_WillsCoupling1}
	\lim_{n\to\infty}\Erw[d_1(Y_\ell(T,\G_n),X_\ell(T,\vec\tau))]=0.
	\end{equation}
\end{lemma}
\begin{proof}
The empirical distribution of marginals   converges in probability to $\int_0^1\atom_{w_x}dx$ (by assumption), and the distribution of the local neighborhood of $x$ converges to $\thet$
(\Prop~\ref{Prop_lwc}), but a priori we do not know how the two distributions are coupled.
To get a handle on this we proceed by induction on the sub-trees of $T$.
Observe that there is nothing to show if $\ell =0$ and  $T$ consists of just the root.  Now by induction, 
let us assume that (\ref{eq_WillsCoupling1}) holds for any depth-$(\ell-1)$ sub-tree of~$T$. 

We need to show that with high probability of the choice of the random factor graph $\G_n$, the empirical distribution $Y_\ell(T,\G_n)$ converges to that of $X_\ell(T,\vec\tau)$.  To accomplish this, we will approximate the moments of $Y_\ell(T,\G_n)$ and show that these converge to those of $X_\ell(T,\vec\tau)$. 

 Consider the following experiment to approximate the first moment of $Y_\ell(T,\G_n)$.  Fix $T, \ell$, and $\eps >0$.   We choose $L = L(\eps)$ large enough (as we will see below), and sample a random factor graph $\G_{n-L}^\prime$ with a slightly sparser constraint density than $\G_{n-L}$: instead of adding a Poisson number of constraints $\psi$ with mean $(n-L) \rho_\psi$, we add a Poisson number with mean $n \rho_\psi \cdot ((n-L)/n)^{k_\psi}$. 
 In particular, $\G_{n-L}^\prime$ has exactly the distribution of the sub-graph of $\G_n$ induced by the first $n-L$ variable nodes.
 Note that the difference in means is constant, much smaller than $\Theta(\sqrt n)$, the standard deviation of the number of each constraint type in $\G_{n-L}$. 
 The distribution of $\G_{n-L}^\prime$ therefore has total variation distance $o(1)$ to $\G_{n-L}$, and in particular, the assumption (\ref{eqRS}) holds for $\G_{n-L}^\prime$.   
 
We next add to $\G_{n-L}^\prime$ $L$ variable nodes $x_{n-L+1}, \dots, x_{n}$, along with a Poisson number of each type of constraint node of mean 
	$n k_\psi\rho_\psi ( 1- ((n-L)/n)^{k_\psi} )$, attached to uniformly random variable nodes from $\G_{n-L}$ and the new variable nodes, \textit{conditioned} on the event that at least one of the attached variables nodes of each new constraint node is one of the newly added variable nodes. The resulting factor graph has exactly the distribution of $\G_{n}$.

Now recalling that the tree $T$ is fixed, we condition on the event that the constraint nodes joined to each of the new variable nodes $x_{n-L+1}, \dots x_{n}$ are of the  number and type joined to the root of $T$, and their attached variable nodes, call them $x_{11}, \dots x_{1k}, \dots x_{L1}, \dots x_{Lk}$, have  depth $\ell-1$ neighborhoods matching those of the attached subtrees $T_1, \dots T_{k}$ of $T$. This event has probability bounded away from $0$ as $n \to \infty$ in $\G_{n+L}$, and so the resulting conditioned graph, $\G_{n}^\prime$, with the last $L$ variables nodes selected, has the distribution of $\G_{n}$ with $L$ variables nodes selected uniformly at  random from all  variable nodes whose depth-$\ell$ neighborhood matches that of $T$.  Let the set $\{ x_{i,a,y} \}$, $i= n-L+1, \dots n$, $a\in\partial_T x_i$, $y \in\partial_{T\downarrow}a$, denote the randomly chosen variable nodes from $\G_n$ which are attached to the new constraint nodes in the given positions. Note that whp $\{ x_{i,a,y} \}$ will only contain variable nodes from $\G_{n-L}^\prime$, as whp no constraint is attached to more than one variable node from a fixed constant-sized set.

 Now  the graph $\G_{n}^\prime$ induces a marginal distribution on the spin at each of the variable nodes $x_{n-L+1}, \dots x_{n}$.  Fix $\omega \in \Omega$, and call the $\omega$ values of these marginals $q_{n-L+1}(\omega), \dots, q_{n}(\omega)$. Let $\overline q(\omega) = \frac{1}{L} \sum_{i=n-L+1}^n q_{i}(\omega)$.  By choosing $L(\eps)$ large enough, we have that whp over the choice of $\G_{n-L}^\prime$ and probability at least $1- \eps$ over the choice of $\{ x_{i,a,y} \}$, $\overline q(\omega)$ is within $\eps$ of the mean of $Y_\ell(T,\G_{n})(\omega)$.  

Similarly, we can approximate the higher moments  of $Y_\ell(T,\G_{n})$, that is, for a vector $\omega_1, \dots, \omega_r \in \Omega^r$ and powers $i_1, \dots i_r$, the mean of $\prod_{j=1}^r  Y_\ell(T,\G_{n})(\omega_j)^{i_j}$.  In this case we again add $L=L(\eps)$ new variable nodes to $\G_{n-L}^\prime$ with the appropriate constraint nodes and attached variable nodes and condition that the local neighborhoods of $x_{n-L+1} \dots x_n$ match $T$ to form $\G_{n}^\prime$. Let $q_{i}(\omega_j)$ denote the marginal probability of $\omega_j$ at variable node $x_i$, $i=n-L+1, \dots , n$. Then let $\overline q = \frac{1}{L} \sum_{i=n-L+1}^{n} \prod_{j=1}^r q_{n+i}(\omega_j)^{i_j}$.  Again by choosing $L = L(\eps)$ large enough we can guarantee that whp over the choice of $\G_{n-L}^\prime$ and probability at least $1- \eps$ over the choice of $\{ x_{i,a,y} \}$, $\overline q$ is within $\eps$ of the corresponding higher moment of $Y_\ell(T,\G_{n})$. 

What remains to show is that these moment calculations converge to those of $X_\ell(T,\vec\tau)$. 
For $\omega_1, \dots, \omega_L \in \Omega$ define
	$$\vec\mu(\omega_1, \dots, \omega_L)=\mu_{\G_{n}^\prime} (\{\sigma(x_{n-L+1}) = \omega_1, \dots, \sigma(x_{n})=\omega_L\});$$
then $\vec\mu$ is a random variable, dependent on $\G_{n-L}'$ and the family $\{ x_{i,a,y} \}$.
Hence, if we condition on $\G_{n-L}'$, then $\vec\mu(\omega_1, \dots, \omega_L)$ remains random, namely dependent on $\{ x_{i,a,y} \}$.
As we saw in the previous paragraph, it suffices to show, for all $\omega_1, \dots, \omega_L \in \Omega$, that 
	$\vec\mu(\omega_1, \dots, \omega_L)$  converges in distribution, with high probability over the choice of $\G_{n-L}^\prime$, to
	$\vec X(\omega_1,\ldots,\omega_L)=\prod_{i=1}^k \tilde X^{(i)}_\ell(T,\vec\tau)(\omega_i)$,
	where the  $\tilde X^{(i)}_\ell(T,\vec\tau)$'s are independent samples from $X_\ell(T,\vec\tau)$.

 Let $Z_{n-L}$ be the partition function of $\G_{n-L}^\prime$, and $Z_{n}$ the partition function of $\G_{n}^\prime$ in the experiment above. Let $\{ x_{i,a,y} \}$, $i= n+1, \dots n_L$, $a\in\partial_T x_i$, $y \in\partial_{T\downarrow}a$, denote the randomly chosen variable nodes from $\G_n$ which are attached to the new constraint nodes in the given positions.  Condition on $\G_{n-L}^\prime$ and the choice of $\{ x_{i,a,y} \}$ and write
\begin{align*}
Z_{n-L} & = \sum_{\vec s \in \Omega^{\{ x_{i,a,y} \}}}  Z_{\vec s}
\end{align*}
where the vector $\vec s$ represents one set of possible values taken by the variable nodes, and $Z_{\vec s}$ is the partition function of $\G_n$ restricted to the set of assignments for which we have $\sigma(x_{i,a,y}) =s_{a,i,y}\, \forall i, a,y$.  
We can write
\begin{align*}
Z_{\vec s} &= Z_{n-L} \cdot \mu_{\G_n} (\sigma( \{ x_{i,a,y}\}) = \vec s).
\end{align*}
Given the family $\{x_{i,a,y}\}$, we have
\begin{align}
\label{eq:marginalFormula}
\vec\mu(\omega_1, \ldots,\omega_L) &=  \frac{    \sum_{\vec s \in \Omega^{\{ x_{i,a,y} \}} }\mu_{\G_{n-L}^\prime} (\sigma( \{ x_{i,a,y}\}) = \vec s) \prod_{a\in\partial_T x_{n-L+1}, \dots \partial_T x_{n}} \psi_a (\omega_i, \vec s_a)  }{ \sum_{\omega^\prime_1, \dots \omega^\prime_L \in \Omega^L} \sum_{\vec s \in \Omega^{\{ x_{i,a,y} \}} }\mu_{\G_{n-L}^\prime} (\sigma( \{ x_{i,a,y}\}) = \vec s) \prod_{a\in\partial_T x_{n-L+1}, \dots \partial_T x_{n}} \psi_a (\omega^\prime_i, \vec s_a) } \,,
\end{align}
where the quantities on both sides are deterministic numbers, as we have conditioned on $\G_{n-L}^\prime$ and the selection of $\{ x_{i,a,y} \}$. 
Now viewing the choice of $\{ x_{i,a,y} \}$ as random, we use the asymptotic factorization property (our assumption (\ref{eqRS}) and \Cor~\ref{Cor_factorise}, recalling that this holds whp for $\G_{n-L}^\prime$ just as it does for $\G_{n-L}$), and get that 
	$\vec\mu(\omega_1,\ldots,\omega_L)$ converges in distribution (whp over the choice of $\G_{n-L}^\prime$) to  \begin{align}
 \label{eqLdistrib}
&\frac{    \sum_{\vec s \in \Omega^{\{ x_{i,a,y} \}} } \prod_{i=1}^L \prod_{a\in\partial_T r_T} \psi_a (\omega_i, \vec s_a) \prod_{y\in\partial_{T\downarrow}a}\tilde X_{i,a,y}(\vec s_{i,a,y})   }{ \sum_{\omega^\prime_1, \dots \omega^\prime_L \in \Omega^L}  \sum_{\vec s \in \Omega^{\{ x_{i,a,y} \}} } \prod_{i=1}^L \prod_{a\in\partial_T r_T} \psi_a (\omega^\prime_i, \vec s_a) \prod_{y\in\partial_{T\downarrow}a}\tilde X_{i,a,y}(\vec s_{i,a,y})  }\,,
\end{align}
where the  $\tilde X_{i,a,y}$'s are independent samples from the respective distributions $X_{\ell-1}(T_{a,y},\vec\tau)$.  We rearrange (\ref{eqLdistrib}) to give the following convergence in distribution, whp over the choice of $\G_n'$:
\begin{align}\label{eqWillLargeL}
\vec\mu(\omega_1,\ldots,\omega_L)
 &\Rightarrow \prod_{i=1}^L \frac{   \prod_{a\in\partial_T r_T} \sum_{\vec s \in \Omega^{k_a}} \mathbf 1_{s_j=\omega_i} \cdot \psi_a (\vec s) \prod_{y\in\partial_{T\downarrow}a}\tilde X_{i,a,y}(\vec s_y) }{  \sum_{\omega^\prime \in \Omega}   \prod_{a\in\partial_T r_T} \sum_{\vec s \in \Omega^{k_a}} \mathbf 1_{s_j=\omega^\prime} \cdot \psi_a (\vec s) \prod_{y\in\partial_{T\downarrow}a}\tilde X_{i,a,y}(\vec s_y)} \,.
\end{align}
Finally, performing the calculation (\ref{eqLdistrib}) in reverse on $L$
independent copies of the tree $T$, we see that
the r.h.s.\ of (\ref{eqWillLargeL}) is distributed as $\vec X(\omega_1,\ldots,\omega_L)$.
In particular this holds for every fixed $L$ and  choice of $\omega_1, \dots \omega_L$, and so this proves convergence of the moments.
\end{proof}

\noindent
To obtain the desired Belief Propagation fixed point we let $\nu_{T,\ell}\in\cP^2(\Omega)$ be the distribution of $X_{\ell}(T,\vec\tau)$.
Moreover, let $\cF_\ell$ be the $\sigma$-algebra on $\fT$ generated by events $\{\partial^\ell\T\ism\partial^\ell T\}$ for $T\in\fT$.

\begin{corollary}\label{Lemma_martingale}
For any $\ell\geq0$ we have $\Erw[\nu_{\T,\ell+1}|\cF_\ell]=\nu_{\T,\ell}$.
\end{corollary}
\begin{proof}
As a first step we are going to show that
	\begin{equation}\label{eqLemma_martingale1}
	\Erw\brk{\nu_{\T,1}|\cF_0}=\Erw\brk{\nu_{\T,1}}=\nu_{\T,0}.
	\end{equation}
Observe that by the definition of $X_0$ the right-hand side is in deterministic.
To prove (\ref{eqLemma_martingale1}) we apply \Lem~\ref{Lemma_WillsCoupling} to $\ell=1$.
Our assumption (\ref{eqRS}) implies that the average empirical distribution $\Erw[Y_1(\T,\G_n)]$ converges to the distribution of $w_{\vec x}$
for a uniform $\vec x\in[0,1)$ as $n\to\infty$.
The latter is precisely the law of $\nu_{\T,0}$.
Moreover, $\Erw[Y_1(\T,\G_n)]$ converges to $\Erw[\nu_{\T,1}]$ by \Lem~\ref{Lemma_WillsCoupling}.
Therefore, $\Erw[\nu_{\T,1}]=\nu_{\T,0}$.

For general values of $\ell$ we proceed by induction.
If we condition on the first $2\ell$ levels of the random tree $\T$, then the trees pending on the variable nodes at distance precisely
$2\ell$ from the root are independent copies of $\T$ itself.
Therefore, the levels $2\ell+1$ and $2\ell+2$ are distributed as $\partial^1\T$.
Hence,  let $V_\ell$ be the set of variable nodes at distance precisely $2\ell$ from the root.
Then (\ref{eqLemma_martingale1}) implies that the distribution 
of each $x\in V_\ell$ under a random boundary condition for $V_{\ell+1}$ as in the construction of $X_{\ell+1}$
is identical to the distribution of $w_{\vec x}$.
Further, all $x\in V_\ell$ are independent.
Consequently, $\Erw[\nu_{\T,\ell+1}|\cF_\ell]=\nu_{\T,\ell}$.
\end{proof}

\begin{proof}[Proof of \Thm~\ref{Thm_BP}]
\Cor~\ref{Lemma_martingale} implies that for any continuous $f:\cP(\Omega)\to\RR$
the sequence $(\int f d\nu_{\T,\ell})_\ell$ is a martingale w.r.t.\ $(\cF_\ell)_\ell$.
Because it is bounded, the martingale converges almost surely and in $L_1$.
Furthermore, because the space of $C(\cP(\Omega))$ of continuous functions on $\cP(\Omega)$ has a countable dense set with respect to uniform convergence
	(e.g., the polynomials with rational coefficients by Weierstrass),
we find that for almost all $\T$ the sequences $(\int f d\nu_{\T,\ell})_\ell$ converge for all $f\in C(\cP(\Omega))$.
Given that this event occurs, the map $\nu_{\T}^\star:f\mapsto\lim_{\ell\to\infty}\int f d\nu_{\T,\ell}$ is a
continuous linear functional on $C(\cP(\Omega))$ that satisfies $\nu_{\T}^\star(1)=1$.
Hence,  by the Riesz representation theorem $\nu_{\T}^\star$ is a probability measure on $\cP(\Omega)$.
Furthermore, our definition (\ref{eqLWC}) of the topology of $\fT$ ensures that for each $\ell$ the function $T\mapsto\nu_{T,\ell}$ is continuous.
Therefore, being a pointwise limit of the continuous functions,  the function $T\mapsto\nu_{\T}^\star$ is measurable.

To establish the fixed point property, we recall that \Lem~\ref{Lemma_WillsCoupling} and (\ref{eqWillLargeL}) imply the following.
For a tree $T$ and a constraint $a\in\partial_Tr_T$ and let
$(\tilde X_{a,y})_y$ be a family of independent copies of $X_{\ell-1}(T_{\downarrow y})$ for $y\in\partial_{T\downarrow}a$.
Then
\begin{align*}
X_{\ell}(T, \vec \tau) &\,\stacksign{d}= \,\frac{   \prod_{a\in\partial_T r_T} \sum_{\vec s \in \Omega^{k_a}} \mathbf 1_{s_j=\omega_i} \cdot \psi_a (\vec s) \prod_{y\in\partial_{T\downarrow}a}\tilde X_{a,y}(\vec s_y) }{  \sum_{\omega^\prime \in \Omega}   \prod_{a\in\partial_T r_T} \sum_{\vec s \in \Omega^{k_a}} \mathbf 1_{s_j=\omega^\prime} \cdot \psi_a (\vec s) \prod_{y\in\partial_{T\downarrow}a}\tilde X_{a,y}(\vec s_y)}
\end{align*}
Hence, taking the limit $\ell\to\infty$, we conclude that $\nu_{\T}^\star$ is a Belief Propagation fixed point.
\end{proof}

\begin{proof}[Proof of \Cor~\ref{cor:marginalneighbors}]
\Cor~\ref{cor:marginalneighbors} follows from the proof above, by taking $L=1$ in the construction, then from (\ref{eq:marginalFormula}) using \Cor~\ref{Cor_factorise} to show that the joint probability $\mu_{\G_{n-1}^\prime} (\sigma( \{ x_{a,y}\}) = \vec s)$ asymptotically factorizes.  Here we do not need to use the fact that the distribution of each marginal converges to the same thing, simply the factorization property. 
\end{proof}

\subsection{Proof of \Thm~\ref{Thm_Bethe}}\label{Sec_Bethe}
Throughout this section we keep the notation and the assumptions of \Thm~\ref{Thm_Bethe}.
Moreover, we let $\vec x=(\vec x_i)_{i\geq1}$ be a sequence of uniform random variables on $[0,1)$ that are mutually independent
and independent of everything else.
We begin with two claims regarding the average marginals of the fixed point and the Bethe free energy.

\begin{claim}\label{Claim_avgMarg}
Let $\vec x\in[0,1)$ be uniform.
Then $\Erw[\nu^\star_{\T}]=\cL(w_{\vec x})$.
\end{claim}
\begin{proof}
Assumption (\ref{eqRS}) implies that the empirical distribution of the marginals of $\G_n$ converges in distribution to $\cL(w_{\vec x})$.
Hence, \Prop~\ref{Prop_lwc} and the assumption that $\nu^\star$ satisfies (\ref{eqThm_BP}) imply the assertion.
\end{proof}

\begin{claim}\label{Lemma_messages}
Let $(\hat{\vec\eta}_{\psi,j,i})_{\psi\in\Psi,j\in[k_\psi],i\geq1}$
	be a family of independent $\cP(\Omega)$-valued random variables with distribution
	\begin{align}\label{eqBetheAlternateA}
	\hat{\vec\eta}_{\psi,j,i}(\sigma_j)&\propto{\sum_{(\sigma_h)_h\in\Omega^{[k_\psi]\setminus\{j\}}}
		\psi(\sigma)\prod_{h\in[k_\psi]\setminus\{j\}}w_{\vec x_h}(\sigma_h)}
	\end{align}
Further, let $(d_{\psi,j})_{\psi\in\Psi,j\in[k_\psi]}$ be a family of independent random variables
such that $d_{\psi,j}$ has distribution $\Po(\rho_\psi/j)$.
Moreover, let $(\vec\eta_h)_{h\geq1}$ be a family of independent $\cP(\Omega)$-valued random variables with distribution
	\begin{align*}
	\vec\eta_h(\sigma)&\propto{\prod_{\psi\in\Psi,j\in[k_\psi],i\in[d_{\psi,j}]}\hat{\vec\eta}_{\psi,j,i}(\sigma)}.
	\end{align*}
Further, let
	\begin{align*}
	\phi&=\ln\sum_{\sigma\in\Omega}\prod_{\psi,j,i}\hat{\vec\eta}_{\psi,j,i}(\sigma),\\
	\hat\phi_{\psi}&=\ln\sum_{\sigma\in\Omega^{k_\psi}}\psi(\sigma)\prod_{h\in[k_\psi]}{\vec\eta}_h(\sigma_h),&
	\tilde\phi_{\psi,j}&=\ln\sum_{\sigma\in\Omega}\vec\eta_1(\sigma)\hat{\vec\eta}_{\psi,j,1}(\sigma)
		&(\psi\in\Psi,\,j\in[k_\psi]).
	\end{align*}
Then
	\begin{align}\label{eqBetheAlternate}
	\cB_\thet(\nu^\star)&=
		\Erw[\phi]+\sum_{\psi\in\Psi}\sum_{j\in[k_\psi]}\rho_\psi\brk{k_\psi^{-1}\Erw[\hat\phi_\psi]-\Erw[\tilde\phi_{\psi,j}]}.
	\end{align}
\end{claim}
\begin{proof}
To show that (\ref{eqBetheAlternate}) and (\ref{eqDefBFE}) coincide, we recall the random variables from (\ref{eqDefBFE1})--(\ref{eqDefBFE3}).
Comparing (\ref{eqDefBFE}) and (\ref{eqBetheAlternate}), we see that it suffices to show
	\begin{align}\label{eqBetheAlternate_1}
	\Erw[\phi]&=\Erw[\varphi_{\T}],&
	\sum_{\psi\in\Psi,j\in[k_\psi]}\rho_\psi\Erw[\hat\phi_{\psi,j}]&=\Erw\brk{\sum_{a\in\partial_{\T}r_{\T}}\hat\varphi_{\T,a}},&
	\sum_{\psi\in\Psi,j\in[k_\psi]}\rho_\psi\Erw[\tilde\phi_{\psi,j}]&=\Erw\brk{\sum_{a\in\partial_{\T}r_{\T}}\tilde\varphi_{\T,a}}.
	\end{align}

To prove the first equality, we recall that by construction the root of the random tree $\T$ has $\Po(\rho_\psi)$ children $a$ with weight function $\psi$.
For each of them the position $j\in[k_\psi]$ such that $\partial_{\T}(a,j)=r_{\T}$ is uniform, and they are independent.
Therefore, for each $j$ the number of $a$ with $\partial_{\T}(a,j)=r_{\T}$ has distribution $\Po(\rho_\psi/k_\psi)$, and these random variables
are independent.
Hence, the distribution of the offspring of the root of $\T$ coincides with the joint distribution of the random variables $(d_{\psi,j})_{\psi,j}$.
Further, for each $a\in\partial_{\T}r_{\T}$ and every $y\in\partial_{\T}a$ the tree $\T_{\downarrow y}$ pending on $y$
is an independent copy of $\T$.
Therefore, Claim~\ref{Claim_avgMarg} implies that the distribution of the random variables $\eta_{\T,y}$
that go into (\ref{eqmsg2}) coincides with the distribution of $w_{\vec x}$ for a uniform $\vec x\in[0,1)$.
Moreover, the random variables $\eta_{\T,y}$ are mutually independent.
Consequently, comparing (\ref{eqBetheAlternateA}) and (\ref{eqmsg2}), we see the distribution of $\hat\eta_{\T,a}$ that for a constraint
node $a$ with $\psi_a=\psi$ and $\partial_{\T}(a,j)=r_{\T}$ coincides with the distribution of $\hat{\vec\eta}_{\psi,j,i}(\sigma_j)$ for any $i\geq1$.
Because the $(\hat\eta_{\T,a})_{a\in\partial_{\T}r_{\T}}$ are mutually independent, we thus see that
$\varphi_{\T}$ has the same distribution as $\phi$.
In particular, $\Erw[\varphi_{\T}]=\Erw[\phi]$.

Further, we derive the middle equation from the Chen-Stein property (\ref{eqmemoryless}).
Specifically,  remembering that we picked one specific representative of each isomorphism class,
we let $a_{\T}$ be a random neighbor of the root of $\T$.
Then
	\begin{align}\label{eqForgetMemoryless}
	\Erw\brk{\sum_{a\in\partial_{\T}r_{\T}}\hat\varphi_{\T,a}}&=
		\Erw\brk{|\partial_{\T}r_{\T}|\hat\varphi_{\T,a_{\T}}}.
	\end{align}
Moreover, by a similar argument as in the previous paragraph the random variables $\eta_{\T,y}$
for $y$ at distance two from $r_{\T}$ are independent copies of $w_{\vec x}$.
Hence, the $\tilde\eta_{\T,a_{\T}}$ factor consists of $|\partial_{\T}r_{\T}|-1$ independent factors.
By comparison, the terms $\vec\eta_{h}$ comprise of $|\partial_{\T}r_{\T}|$ independent factors.
Therefore, applying (\ref{eqmemoryless}) to (\ref{eqForgetMemoryless}), we obtain the middle equation of (\ref{eqBetheAlternate_1}).
The last equation follows form a similar argument.
\end{proof}

\noindent
To derive the theorem from Claim~\ref{Lemma_messages}
we employ the Aizenman-Simms-Starr scheme~\cite{Aizenman}, i.e., the observation that
	\begin{align}\label{eq_Lemma_BFE1}
	\frac1n\Erw[\ln Z_{\G_n}]&=\frac1n\sum_{h=1}^n\Erw\ln\frac{Z_{\G_h}}{Z_{\G_{h-1}}}
	\end{align}
(with the convention that $\G_0$ is the empty graph and $Z_{\G_0}=1$).
The assertion follows from (\ref{eq_Lemma_BFE1}) if we can show that
	\begin{align}\label{eq_Lemma_BFE2}
	\lim_{n\to\infty}\Erw\ln\frac{Z_{\G_{n+1}}}{Z_{\G_n}}&=\cB_\thet(\nu^\star).
	\end{align}

Hence, we need to compare $\G_{n+1},\G_n$ for large $n$.
To this end, we couple the two random graphs 
as follows.
Let $\rho_\psi'=(n/(n+1))^{k_\psi}\rho_\psi$ and let
 $\G'$ be a random factor graph on the variable set $V_n$ obtained by including $m_\psi'=\Po(\rho_\psi' n)$ random constraints of type $\psi$  for each $\psi\in\Psi$.
Then the distribution of $\G'$ coincides with the distribution of $\G_{n+1}$ with variable $x_{n+1}$ and all incident constraints removed.
Further, obtain  $\G''$ by adding $m_{\psi}''=\Po((\rho_\psi-\rho_{\psi}')n)$ random constraints to $\G'$ independently for each $\psi\in\Psi$.
Then $\G''$ is distributed as $\G_n$.
Finally, obtain $\G'''$ from $\G'$ by adding one variable $x_{n+1}$ and $m_\psi'''=\Po(k_\psi\rho_\psi)$ random constraints of type $\psi$
that are incident to $x_{n+1}$ independently for each $\psi\in\Psi$. 
Then the distribution of $\G'''$ matches that of $\G_{n+1}$.

To calculate $\Erw\ln (Z_{\G''}/Z_{\G'})$, fix some $\eps>0$.
There is a number $c>0$ such that $\max\{|\ln\psi(\sigma)|:\psi\in\Psi,\sigma\in\Omega^{k_\psi}\}<c$.
Hence, because the tails of the Poisson decay sub-exponentially, there exists $L=L(\eps)>0$ such that
	\begin{align}\label{eq_Lemma_BFE3}
	\textstyle\Erw\brk{\vecone\{\sum_{\psi} m_\psi'''>L\}|\ln (Z_{\G'''}/Z_{\G'})|}&<\eps.
	\end{align}
We are now going to insert the new constraints into $\G'$ one by one and track the impact of each insertion.
Number the new constraints in some way as $a_1'',\ldots,a_l''$ with $l\leq L$, let $\G''_0=\G'$ and let
$\G''_i$ be the obtained by inserting $a_1'',\ldots,a_i''$.
Because the total number of constraints that we add is of lower order than the standard variation $\Theta(\sqrt n)$ of the number of constraints of each type,
the total variation distance of $\G''_i$ and $\G_n$ is $o(1)$ for each $i\leq L$.
Therefore, if we let $(x_{ij})_{j\in[k_i]}$ be the family of independent, uniformly chosen neighbors of $a_i''$, then
\Cor~\ref{Cor_factorise} and our assumption (\ref{eqRS}) imply that with high probability
	\begin{align*}
	\TV{\mu_{\G''_{i-1}\marg \{x_{ij}:j\in[k_i]\}}-\bigotimes_{j\in[k_i]}\mu_{\G''_{i-1}\marg x_{ij}}}=o(1)\qquad\mbox{as }n\to\infty.
	\end{align*}
Together with (\ref{eqRS}) 
this implies that the vector $(\mu_{\G'_{i-1}\marg x_{ij}})_{j\in[k_\psi]}$ converges in distribution to $(w_{\vec x_j})_{j\in[k_\psi]}$ 
for independent uniform $\vec x_j\in[0,1)$.
Hence, 
	\begin{align}\label{eq_Lemma_BFE5}
	\Erw\ln\bc{Z_{\G''_i}/Z_{\G''_{i-1}}}&=\Erw\brk{\ln\sum_{\sigma\in\Omega^{k_i}}\psi_{a_i''}(\sigma)\prod_{j\in[k_1]}w_{\vec x_j}(\sigma_j)}+o(1).
	\end{align}
Summing (\ref{eq_Lemma_BFE5}) up over $l\leq L$, using (\ref{eq_Lemma_BFE3}) and letting $\eps\to0$ sufficiently slowly, we obtain
	\begin{align}\label{eq_Lemma_BFE6}
	\Erw\ln\bc{Z_{\G''}/Z_{\G'}}&=\sum_{\psi\in\Psi}(\rho_\psi-\rho_\psi')
		\Erw\brk{\ln\sum_{\sigma\in\Omega^{k_\psi}}\psi(\sigma)\prod_{j\in[k_\psi]}w_{\vec x_j}(\sigma_j)}+o(1).
	\end{align}
Finally, in the notation of Claim~\ref{Lemma_messages}, (\ref{eq_Lemma_BFE6}) reads
	\begin{align}\label{eq_Lemma_BFE6a}
	\Erw\ln\bc{Z_{\G'''}/Z_{\G'}}&=\sum_{\psi\in\Psi}(\rho_\psi-\rho_\psi')
		\Erw\brk{\hat\phi_\psi}+o(1).
	\end{align}

To calculate $\Erw\ln\bc{Z_{\G'''}/Z_{\G'}}$ let $a_1,\ldots,a_l$ be the new constraints attached to $x_{n+1}$.
With probability $1-O(1/n)$ the new variable $x_{n+1}$ appears precisely once in each $a_i$.
If so, then by the same token as in the previous paragraph the joint distribution $\mu_{\G'\marg Y}$ of the variables
	$Y=\bigcup_{i\leq l}\partial_{\G'''}a_i\setminus\{x_{n+1}\}$ converges to 
	$(w_{\vec x_y})_{y\in Y}$ with independent uniform $\vec x_y\in[0,1)$.
Consequently,
	\begin{align}\label{eq_Lemma_BFE7}
	\Erw\ln\bc{Z_{\G'''}/Z_{\G'}}&=o(1)+\Erw\brk{\ln\sum_{\sigma
	\in\Omega^{\{x_{n+1}\}\cup Y}}
    \prod_{i=1}^l\bc{\psi_{a_i}((\sigma_y)_{y\in\partial_{\G'''}a_i})\prod_{y\in Y\cap  
    \partial_{\G'''}a_i} w_{\vec x_y}(\sigma_y)}}.
	\end{align}
We introduce probability distributions on $\Omega$ by letting
	\begin{align*}
	\hat\nu_i(\sigma_{x_{n+1}})&\propto\sum_{(\sigma_y)_y\in\Omega^{\partial_{\G'''} a_i\setminus\{x_{n+1}\}}}
		\psi_{a_i}(\sigma)\prod_{y\in\partial_{\G'''} a_i\setminus\{x_{n+1}\}}w_{\vec x_y}(\sigma_y),&
	\nu(\sigma_{x_{n+1}})&\propto\prod_{i=1}^l\hat\nu_i(\sigma_{x_{n+1}})\qquad(\sigma_{x_{n+1}}\in\Omega).
	\end{align*}
Then (\ref{eq_Lemma_BFE7}) becomes
	\begin{align}\nonumber
	\Erw\ln\bc{Z_{\G'''}/Z_{\G'}}+o(1)&=
		\Erw\brk{\ln\sum_{\sigma_{x_{n+1}}}\prod_{i=1}^l\brk{\hat\nu_i(\sigma_{x_{n+1}})
				\sum_{\tau_{x_{n+1}}}\hat\nu_i(\tau_{x_{n+1}})}}\\
		&=\Erw\brk{\ln\sum_{\sigma_{x_{n+1}}}\prod_{i=1}^l\hat\nu_i(\sigma_{x_{n+1}})}
			+\sum_{i=1}^l\Erw\brk{\ln\sum_{\tau_{x_{n+1}}}\hat\nu_i(\tau_{x_{n+1}})}.
				\label{eq_Lemma_BFE9}
	\end{align}
Further,
	\begin{align}
	\ln\sum_{\tau_{x_{n+1}}}\hat\nu_i(\tau_{x_{n+1}})&=
		\ln\brk{\sum_{\tau\in\Omega^{\partial_{\G'''}a_i}}\psi_{a_i}(\tau)\nu(\tau_{x_{n+1}})
			\prod_{y\in\partial_{\G'''}a_i\setminus\{x_{n+1}\}} w_{\vec x_y}(\tau_y)}
		-\ln\sum_{\tau_{x_{n+1}}}\nu(\tau_{x_{n+1}})\hat\nu_i(\tau_{x_{n+1}}).
		\label{eq_Lemma_BFE10}
	\end{align}
To connect these formulas with \Lem~\ref{Lemma_messages}, we observe that
$\sum_{\sigma_{x_{n+1}}}\prod_{i=1}^l\hat\nu_i(\sigma_{x_{n+1}})$ is distributed as $\phi$ from \Lem~\ref{Lemma_messages}.
Moreover, by the Chen-Stein property we have
	\begin{align*}
	\Erw\brk{\sum_{i=1}^l\ln{\sum_{\tau}\psi_{a_i}(\tau)\nu(\tau_{x_{n+1}})\prod_{y\in\partial_{\G'''}a_i} w_{\vec x_y}(\tau_y)}}
		&=\sum_{\psi\in\Psi}k_\psi\rho_\psi\Erw[\hat\phi_{\psi}],\\
	\Erw\brk{\sum_{i=1}^l\ln\sum_{\tau_{x_{n+1}}}\nu(\tau_{x_{n+1}})\hat\nu_i(\tau_{x_{n+1}})}&=
		\sum_{\psi\in\Psi}\sum_{j\in[k_\psi]}\rho_\psi\Erw[\tilde\phi_{\psi,j}].
	\end{align*}
Therefore, 
	\begin{align}		\label{eq_Lemma_BFE11}
	\Erw\ln\bc{Z_{\G'''}/Z_{\G'}}&=\Erw[\phi]+\sum_{\psi\in\Psi}
	    k_\psi\rho_\psi\Erw[\hat\phi_\psi]
		+\sum_{\psi\in\Psi,j\in[k_\psi]}\rho_\psi\Erw[\tilde\phi_{\psi,j}]
		+o(1).
	\end{align}
Finally, combining (\ref{eq_Lemma_BFE6a}) and (\ref{eq_Lemma_BFE11}),
we obtain (\ref{eq_Lemma_BFE2}).

\end{document}